\documentclass[12pt]{amsart}

\usepackage[dvips]{graphics}

\usepackage{epsfig}
\usepackage{stmaryrd}

\usepackage{amsfonts}
\usepackage{amsmath, amsthm, amssymb, ulem, amscd} 

\newcommand{\stopthm}{\hfill$\square$\medskip}

\newcommand{\cF}{{\mathcal F}} 
\newcommand{\cP}{{\mathcal P}} 
\newcommand{\pa}{\partial}
\newcommand{\Up}{\Upsilon}
\newcommand{\Om}{\Omega}
\newcommand{\La}{\Lambda}

\newcommand{\ep}{\epsilon}

\newcommand{\ga}{\gamma}  
\newcommand{\om}{\omega}  
\newcommand{\de}{\delta}  
\newcommand{\db}{\underline{\delta}}
\newcommand{\sigb}{\overline{\sigma}}
\newcommand{\De}{\Delta}  
\newcommand{\na}{\nabla}  
\newcommand{\p}{\varphi}  
\newcommand{\ph}{\widehat{\varphi}}  

\newcommand{\Gat}{{\tilde \Gamma}}

\newcommand{\tr}{\operatorname{tr}}

\newcommand{\cG}{{\mathcal G}}

\newcommand{\cO}{{\mathcal O}}
\newcommand{\cQ}{{\mathcal Q}}

\newcommand{\cV}{{\mathcal V}}

\newcommand{\cT}{{\mathcal T}}
\newcommand{\cS}{{\mathcal S}}

\newcommand{\cL}{{\mathcal L}}

\newcommand{\cJ}{{\mathcal J}}  

\newcommand{\gt}{{\tilde g}}
 
\newcommand{\gh}{{\widehat g}} 
\newcommand{\hh}{{\widehat h}}

\newcommand{\Rt}{{\tilde R}}

\newcommand{\R}{{\mathbb R}}

\newcommand{\nf}{\infty}

\newcommand{\Vol}{\operatorname{Vol}}
\newcommand{\Ric}{\operatorname{Ric}}

\renewcommand{\tilde}{\widetilde} 
\newcommand{\wh}{\widehat}

\theoremstyle{plain}
\newtheorem{theorem}{Theorem}[section]

\newtheorem{proposition}[theorem]{Proposition}
\newtheorem{corollary}[theorem]{Corollary}

\theoremstyle{definition}

\newtheorem{definition}[theorem]{Definition}

\theoremstyle{remark}

\numberwithin{equation}{section}

\title[Renormalized Volume Coefficients]{Extended Obstruction Tensors and
  Renormalized Volume Coefficients}     

\author{C. Robin Graham}
\address{Department of Mathematics, University of Washington,
Box 354350\\
Seattle, WA 98195-4350}
\email{robin@math.washington.edu}

\begin{document}

\maketitle

\thispagestyle{empty}

\renewcommand{\thefootnote}{}
\footnotetext{Partially supported by NSF grant \# DMS 0505701.}   
\renewcommand{\thefootnote}{1}

\section{Introduction}\label{intro}
In recent years there has been a great deal of progress on   
the so-called $\sigma_k$-Yamabe problem.  In \cite{CF}, Alice Chang and Hao
Fang have suggested that a variant of this problem might also be fruitful
to study.  The main goal of this paper is to investigate the algebraic  
structure under conformal transformation of the renormalized volume 
coefficients, the curvature quantites considered by Chang-Fang.  
A key ingredient in the investigation is the introduction of ``extended 
obstruction tensors'', which are anticipated to be of independent  
interest.  These are natural tensors associated to a  
pseudo-Riemannian metric $g$ which turn out to be building blocks for the   
expansion of the ambient or Poincar\'e metric determined by $g$ and thus
also for the renormalized volume coefficients.     

The $\sigma_k$-Yamabe problem was introduced by Jeff Viaclovsky in
\cite{V}.  Let 
$$
P_{ij}=\frac{1}{n-2}\left(R_{ij}-\frac{R}{2(n-1)}g_{ij}\right)
$$
denote the Schouten tensor of a metric $g$ on a manifold $M$ of 
dimension $n\geq 3$, and let $g^{-1}P$ denote the endomorphism 
$P^i{}_j$ obtained by raising an index.  For $1\leq k\leq n$, the
$\sigma_k$-Yamabe problem 
is to find a metric in a given conformal class for which 
$\sigma_k(g^{-1}P)$ is constant, where 
$\sigma_k(A)$ denotes the \mbox{$k$-th} 
elementary symmetric function of the eigenvalues of an endomorphism $A$.
We set $\sigma_k=0$ for $k>n$.    
For $k=1$, $\sigma_1(g^{-1}P)$ is a multiple of the scalar curvature of
$g$, so this is the Yamabe problem.  For $2\leq k\leq n$,  
$\sigma_k(g^{-1}P)$ is a second order fully nonlinear operator in
the conformal factor.  

Variational methods have played an important role in the study of this
problem.  In dimensions $n>2$, the   
Yamabe equation $R=c$ is the Euler-Lagrange equation for the total
scalar 
curvature functional $\int_M R\,dv_g$ under conformal variations subject to
the constraint $\Vol_g(M)=1$.  Of course, this fails when $n=2$ because 
of the Gauss-Bonnet Theorem.  For $k=2$ the analogous special dimension 
is $n=4$.  In this dimension, the total $\sigma_2$ curvature 
$\int_M\sigma_2(g^{-1}P)\,dv_g$ is a conformal invariant.
If $n\neq 4$, the natural generalization of the   
variational characterization of the Yamabe equation holds:   
the equation $\sigma_2(g^{-1}P)=c$ is 
the Euler-Lagrange equation for the 
functional $\int_M \sigma_2(g^{-1}P)\,dv_g$ under conformal variations
subject to the constraint $\Vol_g(M)=1$.  

Both of these properties fail for general metrics when $3\leq k\leq n$.   
The special dimension is now $n=2k$.  But it is no longer true that 
$\int_M \sigma_k(g^{-1}P)\,dv_g$ is a conformal invariant in dimension
$2k$.  Nor is it true  for $n\neq 2k$ that the equation  
$\sigma_k(g^{-1}P)=c$ is the constrainted Euler-Lagrange equation
for the total $\sigma_k$ functional.   
Viaclovsky did show that  
both properties hold if $g$ is locally conformally flat.  
But Branson and Gover proved in \cite{BG} that if $3\leq k\leq n$  
and $g$ is not locally conformally flat, then the equation 
$\sigma_k(g^{-1}P)=c$ is not the Euler-Lagrange equation of any
functional.     

The renormalized volume coefficients of $g$, denoted here by $v_k(g)$,  
arose in the late '90's in the physics literature in the context of
the AdS/CFT correspondence.  A mathematical discussion is contained in 
\cite{G}.  They are defined in terms of the
expansion of the ambient or Poincar\'e metric associated to $g$ in the
sense of \cite{FG1}.  One searches for a smooth 1-parameter family 
of metrics $h_r$ on $M$ so that $h_0=g$ and so that the metric 
\begin{equation}\label{gplus}
g_+=\frac{dr^2+h_r}{r^2}
\end{equation} 
on $M\times (0,\epsilon)$ is an asymptotic solution to $\Ric(g_+)=-ng_+$ at
$r=0$.  This together with the condition that $h_r$ be even in $r$ uniquely
determines $h_r$ to infinite order if $n$ is odd, however only to order
$n$ if $n$ is even, at which point there is a formal obstruction to finding
a solution to the next order.  The trace part of the Taylor coefficient at
order $n$ is determined but the determination of the trace-free part is
obstructed by a trace-free symmetric 2-tensor called the ambient
obstruction tensor.  

Since $h_r$ is even in $r$, it is natural to introduce a new variable   
$\rho = -\frac12 r^2$ and set $g_\rho = h_r$.  The ambient metric 
coefficients are the determined Taylor coefficients
$\pa_\rho^kg_\rho|_{\rho=0}$. 
These are given locally in terms of the initial metric $g_0=g$; each of
them can be written as a  polynomial natural  
tensor expressible in terms of the curvature tensor of $g$ and its
covariant derivatives.   The renormalized volume coefficients are defined
by the expansion of the volume form: 
\begin{equation}\label{ambexpansion}
\left(\frac{\det g_\rho}{\det g_0}\right)^{1/2} 
\sim 1+\sum_{k=1}^\infty v_k\rho^{k}.  
\end{equation}
If $n$ is odd, $v_k(g)$ is defined for all $k\geq 1$.  If $n$  
is even, $v_k(g)$ is defined only for $k\leq n/2$ for general $g$, although  
$v_k(g)$ is defined for all $k\geq 1$ also in even dimensions if $g$ is
Einstein or 
locally conformally flat.  A more detailed discussion is contained in
\S\ref{eot}. 

The insight of Chang-Fang is to consider the $v_k(g)$ in the context of the  
properties satisfied by the $\sigma_k(g^{-1}P)$.  
Just comparing the formulae for these quantities shows that
$v_k(g)=\sigma_k(g^{-1}P)$ if $k=1$ or 2.  In \cite{GJ} it is shown 
that this holds also for $k\geq 3$ if $g$ is  
locally conformally flat.  Moreover,      
$v_k(g)$ always satisfies the two properties discussed above which failed 
for $\sigma_k(g^{-1}P)$ for $3\leq k \leq n$ for general metrics.  One of
the first important 
properties established of the $v_k$ was that in dimension $n=2k$,
$\int_Mv_k(g)\,dv_g$ is a conformal 
invariant for general metrics (a proof is given in \cite{G}).  And the new
result of 
Chang-Fang is that the variational characterization holds for
$v_k(g)$:  for $n\neq 2k$, the equation $v_k(g)=c$ is     
the Euler-Lagrange equation for the functional $\int_Mv_k(g)\,dv_g$ under 
conformal variations subject to the constraint $\Vol_g(M)=1$.  This
collection of facts suggests a strong parallel between the $v_k(g)$ and the  
$\sigma_k(g^{-1}P)$, and even that from some points of view the $v_k(g)$ 
have better properties.  

However, study of the $v_k(g)$ involves significant challenges not shared
by the $\sigma_k(g^{-1}P)$.  Firstly, for $k\geq 3$, $v_k(g)$ depends on   
derivatives of the curvature of $g$.  In fact, for $k\geq 2$, $v_k(g)$
depends on derivatives of curvature of order up to $2k-4$.  Secondly, 
the $v_k(g)$ are defined via an indirect, highly nonlinear, inductive 
algorithm:  first one solves the Einstein equation formally to determine 
$g_\rho$ and then expands its volume form to obtain $v_k(g)$.  
They are algebraically complicated and no explicit formula is known for
general $k$.      

A formula for $v_3$ was given in \cite{GJ}; it is not difficult to carry
out the algorithm explicitly by hand to this order.  The result is:  
\begin{equation}\label{v3}
v_3(g)=\sigma_3(g^{-1}P)+\frac{1}{3(n-4)}P^{ij}B_{ij},
\end{equation}
where $B_{ij}$ denotes the Bach tensor of $g$.  
It is well-known that 
under conformal change $\gh=e^{2\om}g$, the transformation law of the Bach
tensor involves just first derivatives of the conformal factor.  Thus an 
immediate consequence of \eqref{v3} and the conformal 
tranformation law   
\begin{equation}\label{Ptransform}
\widehat{P}_{ij}=P_{ij}-\om_{ij} +\om_i\om_j-\tfrac12 \om_k\om^kg_{ij}
\end{equation}
of the Schouten tensor is the fact that the transformation law of
$v_3$ involves at most second order derivatives of $\om$.  Thus for a fixed 
metric $g$, the equation $v_3(e^{2\om}g)=c$ is second order in $\om$.  It
is this equation that Chang-Fang propose to study by analogy with the 
$\sigma_k$-Yamabe problem.  

In this paper, it is proved that the conformal transformation law involves
at most second order derivatives of $\om$ for all the $v_k$, as well as for
all the ambient metric coefficients.    
\begin{theorem}\label{atmost2}
Under conformal change $\gh=e^{2\om}g$, the conformal transformation laws
of the $\pa_\rho^kg_\rho|_{\rho=0}$ and  
the $v_k$ involve at most second derivatives of $\om$.   If $n$ is odd, 
this is true for all $k$.  If $n$ is even, it is true for  
$\pa_\rho^kg_{ij}|_{\rho=0}$ for $1\leq k\leq n/2-1$, and for     
$g^{ij}\pa_\rho^{n/2}g_{ij}|_{\rho=0}$ and $v_k$ for $1\leq k\leq n/2$.        
\end{theorem}

We give two different proofs of Theorem~\ref{atmost2}, each of which yields
further information.  The first proof proceeds by establishing that each of
the 
determined ambient metric coefficients $\pa_\rho^kg_{ij}|_{\rho=0}$ 
can be written in terms of simpler building blocks, each of which has a 
conformal transformation law involving at most second derivatives of $\om$.
The building blocks consist of the Schouten tensor and a family
$\Om^{(k)}_{ij}$ of trace-free symmetric natural 2-tensors 
which we call the extended obstruction tensors.  The $\Om^{(k)}_{ij}$ 
are defined for all $k\geq 1$ if $n$ is odd, but only for 
$1\leq k\leq n/2-2$ if $n$ is even.  The name derives from the fact that 
when the
dimension is viewed as a formal parameter, $\Om^{(k)}_{ij}$ 
has a simple pole at dimension $n=2k+2$ whose residue is a multiple of the  
obstruction tensor in that dimension.  For example, 
\begin{equation}\label{omega1}
\Om^{(1)}_{ij}=\frac{1}{4-n}B_{ij},
\end{equation}
and the obstruction tensor in dimension 4 is the Bach tensor $B_{ij}$.    
The result asserting that the ambient metric coefficients can be written in
terms of the building blocks is the following.     
\begin{theorem}\label{coefficientform}
Let $k\geq 1$.  There is a linear combination 
$\cG^{(k)}_{ij}\left(P, \Om^{(1)},\ldots,\Om^{(k-1)}\right)$
of partial contractions with respect to $g^{-1}$ of 
the Schouten tensor $P$   
and the $\Om^{(l)}$, $1\leq l\leq k-1$, such that the coefficients of
$\cG^{(k)}_{ij}$ are independent of $n$, and such that 
the ambient metric coefficients in dimension $n$ are given by:
\begin{equation}\label{derivform}
\pa_\rho^kg_{ij}|_{\rho=0} = \cG^{(k)}_{ij}\left(P,  
\Om^{(1)},\ldots,\Om^{(k-1)}\right), 
\end{equation}
for all $k\geq 1$ if $n$ is odd and for $1\leq k\leq n/2-1$ if $n$ is   
even.  Additionally, if $k\geq 2$, there is a linear combination 
$\cT^{(k)}\left(P, \Om^{(1)},\ldots,\Om^{(k-2)}\right)$ of complete
contractions of the indicated tensors whose coefficients are independent of
$n$, such that in even dimension $n$, one has 
\begin{equation}\label{trace}
g^{ij}\pa_\rho^{n/2}g_{ij}|_{\rho=0} = \cT^{(n/2)}\left(P,   
\Om^{(1)},\ldots,\Om^{(n/2-2)}\right).
\end{equation}
\end{theorem}
A corollary is the analogous result for the renormalized volume 
coefficients.  
\begin{corollary}\label{vform}
Let $k\geq 1$.  There is a linear combination 
$\cV_k\left(P, \Om^{(1)},\ldots,\Om^{(k-2)}\right)$ 
of complete contractions with respect to $g^{-1}$ of  
the Schouten tensor $P$   
and the $\Om^{(l)}$, $1\leq l\leq k-2$, such that the coefficients of 
$\cV_k$ are independent of $n$, and such that the renormalized volume
coefficients in dimension $n$ are given by 
$$
v_k(g) = \cV_k\left(P, \Om^{(1)},\ldots,\Om^{(k-2)}\right), 
$$
for all $k\geq 1$ if $n$ is odd and for $1\leq k\leq n/2$ if $n$ is     
even.
\end{corollary}

\noindent
For example, \eqref{v3} and \eqref{omega1} give
$$
v_3(g)=\sigma_3(g^{-1}P)-\tfrac13 P^{ij}\Om^{(1)}_{ij}.
$$

The proof of Theorem~\ref{coefficientform} gives a fairly simple, direct
algorithm for the inductive determination of the $\cG^{(k)}_{ij}$ which is
independent of the formal solution of the Einstein equation.  
It is easy to carry this out to exhibit $\cG^{(k)}_{ij}$ for small $k$; we  
give the result for $k\leq 5$.  The more significant algebraic
complexity occurs in the  
expressions for the $\Om^{(k)}_{ij}$ in terms of the curvature of $g$, for  
which solution of  the Einstein equation is required and in which
the dimension enters explicitly.   

The extended obstruction tensors are part of the theory of conformal
curvature tensors developed in \S6 of \cite{FG2}; they are particular
instances of conformal 
curvature tensors.  In particular, each of them has the property shared by
all conformal curvature tensors that its conformal transformation law  
can be written explicitly in terms of other conformal curvature tensors and 
first derivatives of the conformal factor.  Thus Theorem~\ref{atmost2}
follows immediately from Theorem~\ref{coefficientform} and
Corollary~\ref{vform}.  Moreover, this shows that the only way second 
derivative terms in $\om$ can arise in the conformal transformation law of
$v_k(g)$ is from occurrences in $\cV_k$ of the Schouten tensor. 

A closer analysis of the conformal 
transformation law of the $\Om^{(k)}_{ij}$ and of the form of the $\cV_k$
gives  
the following result describing the structure of $v_k(e^{2\om}g)$ as a
second order fully nonlinear operator.
\begin{theorem}\label{structure}
Let $k\geq 1$ and suppose $k\leq n/2$ if $n$ is even.  Then   
$$
e^{2k\om} v_k(\gh)= \sigma_k(g^{-1}\wh{P})+
\sum_{m=0}^{k-2}r_{k,m}(x,\nabla\om,\widehat{P}),
$$
where $r_{k,m}(x,\nabla\om,\widehat{P})$ is a polynomial 
in $(\om_i$, $\widehat{P}_{ij})$ which is homogeneous of  
degree 
$m$ in $\widehat{P}$, of degree $\leq 2k-2m-2$ in $\nabla\om$, and with  
coefficients depending on $g$.  
\end{theorem}

\noindent
Here $\wh{P}$ is the conformally
transformed Schouten tensor given by \eqref{Ptransform}.

Theorem~\ref{structure} shows that $\sigma_k(g^{-1}\wh{P})$ can be viewed
as the leading term 
in $v_k(\gh)$ from two points of view (at least for $k\leq n$ so that
$\sigma_k\neq 0$).  First, it has the highest    
homogeneity degree in $\widehat{P}$, which contains all the second 
derivative terms.  Second, it contains all the terms with the highest   
total number of derivatives of $\om$.  By this we mean that we expand 
$e^{2k\om}v_k(\gh)$ as a polynomial in $\om_i$, $\om_{ij}$ and add up 
the total number derivatives on $\om$ in each monomial.  For example, in  
\eqref{Ptransform}, each of the terms $\om_{ij}$, $\om_i\om_j$ and
$\om_k\om^k$ has a total of 2 derivatives of $\om$.  Thus 
$\sigma_k(g^{-1}\wh{P})$ contains terms with $2k$ derivatives of $\om$. 
Theorem~\ref{structure} implies that each term in each 
$r_{k,m}(x,\nabla\om,\widehat{P})$ involves at most $2k-2$ derivatives of
$\om$.  It is tempting to speculate that requiring these two properties
gives a reasonable definition of a ``principal part'' of a second order
fully nonlinear operator depending polynomially on the derivatives.  

The second proof of Theorem~\ref{atmost2} proceeds via a study of  
the linearization of $v_k(e^{2\om}g)$ as a function of $\om$, i.e. of the
linearized conformal transformation law of $v_k(g)$.  The main 
ingredient is a formula for the conformal variation of the 1-parameter
family $h_r$ of metrics on $M$ which arise when a given asymptotically 
hyperbolic metric $g_+$ is written in the form 
\eqref{gplus}.  If one chooses a conformally related boundary metric
$\gh=e^{2\om}g$, then up to a 
diffeomorphism of $M\times [0,\ep)$, $g_+$ can be written in the form
\eqref{gplus} with a different 1-parameter family 
$\hh_r$ satisfying $\hh_0=\gh$.  
It is possible to solve explicitly for the infinitesimal diffeomorphism  
in terms of $\om$ and then for the infinitesimal conformal variation of
$h_r$.  Letting $\delta$ denote infinitesimal conformal variation, the
result 
when written in terms of $g_\rho=h_r$ defined as above is the following:   
(it is convenient to use also the notation $g_{ij}(\rho)$ or simply
$g_{ij}$ for $g_\rho$)  
\begin{equation}\label{newlaw}
(\delta g)_{ij} = 2\om( 1 -\rho\pa_\rho)g_{ij} +2 \nabla_{(i}Y_{j)}
\end{equation}
with
\begin{equation}\label{Yformula}
Y^i(\rho)=-\int_0^\rho g^{ij}(u)\,du\; \pa_j\om, \qquad 
Y_j(\rho)=g_{ij}(\rho)Y^i(\rho).
\end{equation}
Here $\nabla_i$ denotes the covariant derivative with respect to  
$g_\rho$ with $\rho$ fixed.  An easy consequence of this is a
formula for the infinitesimal conformal variation of the $v_k(g)$.  Set  
\begin{equation}\label{vdef}
v(\rho)=\left(\frac{\det g_\rho}{\det g_0}\right)^{1/2}.  
\end{equation}
\begin{theorem}\label{L}
Let $k\geq 1$ and $k\leq n/2$ if $n$ is even.  The infinitesimal 
conformal variation of $v_k$ is given by:  
\begin{equation}\label{dvkform}
\delta v_k = -2k\om v_k +\nabla_i\left(L^{ij}_{(k)}\nabla_j \om\right),
\end{equation}
where
\begin{equation}\label{Lformula}
L^{ij}_{(k)}=
-\frac{1}{k!}\pa_\rho^k\left( v(\rho)\int_0^\rho g^{ij}(u)\,du
\right)\Big{|}_{\rho =0}
=-\sum_{l=1}^k\frac{1}{l!}v_{k-l}\,\pa_\rho^{l-1}g^{ij}|_{\rho =0}.  
\end{equation}
In \eqref{dvkform}, $\nabla_i$ denotes the covariant derivative with
respect to the initial metric $g=g_0$.  
\end{theorem}

The infinitesimal transformation laws \eqref{newlaw} and \eqref{dvkform}
clearly involve 
derivatives of $\om$ of order at most 2.  The second proof of
Theorem~\ref{atmost2} proceeds by arguing that if the infinitesimal
conformal transformation law of a natural tensor involves at most $m$  
derivatives of $\om$ for some $m\geq 0$, then the same is true of the full
transformation law.  This is the content of Proposition~\ref{linfull}. 

The Chang-Fang variational characterization of the equations $v_k(g)=c$ as
Euler-Lagrange equations if $n\neq 2k$ is an easy consequence of
\eqref{dvkform}.  
The main point is that the second term on the right hand side of 
\eqref{dvkform} is a divergence, which 
integrates to zero.  Thus the only contribution to the Euler-Lagrange
equation is the scaling contribution given by the first term, so that the 
Euler-Lagrange equation is $v_k(g)=c$.  
The proof of Chang-Fang is also based on using the diffeomorphism
invariance of $g_+$ under conformal change of $g$ to generate a divergence
term; these amount to different versions of the same proof.  But by working  
directly with metrics in the normal form \eqref{gplus}, it is possible  
to give explicit formulae for the divergence terms, among other things
making it clear that 
these terms depend on no more than second derivatives of $\om$.  The 
approach to the linearization formulae used here is the same as in 
\cite{ISTY}, where the formulae \eqref{newlaw} and  \eqref{dvkform} already 
appear.        

Formula \eqref{dvkform} can be interpreted as identifying the 
linearization at $\om=0$ of     
the second order fully nonlinear operator $v_k(e^{2\om}g)$  
with $g$ fixed.  In particular, the linearization is exhibited in
divergence form modulo the zeroth order scaling term, and its principal
part is $L^{ij}_{(k)}\nabla^2_{ij}\om$.  This may be 
useful in determining ellipticity of $v_k(e^{2\om}g)=c$.  However,  
although all of these linearization formulae are explicit, they are written
in terms of the coefficients in the 
ambient metric and renormalized volume expansions, and therefore are  
difficult to understand directly in terms of geometry of $g$.
Additionally, they involve more and more derivatives of $g$ as $k$
increases. 

The results obtained in this paper support and extend the suggestion of
Chang-Fang  
that the $v_k(g)$ are worthy of further study.  However, significant
algebraic 
complications remain and the geometric content of the
equations $v_k(e^{2\om}g)=c$ is unclear, particularly for large $k$.
Perhaps it would be reasonable to try to extend directly the analytic
theory of the    
$\sigma_k$ equations to elliptic fully nonlinear equations    
allowing ``lower order terms'' with structure as in
Theorem~\ref{structure}.  

\section{Extended Obstruction Tensors}\label{eot}
We begin this section by recalling the Poincar\'e metric expansion 
and the definition of the renormalized volume    
coefficients.  These are then reformulated in terms of the expansion of the  
ambient metric.  After reviewing the theory of
conformal curvature tensors from \cite{FG2}, we define the extended
obstruction tensors as certain specific conformal curvature tensors.   
We establish the basic properties of the extended obstruction tensors.
Then we prove Theorems~\ref{coefficientform} and \ref{structure}.  The
section is concluded by giving some explicit formulae for small $k$.       

First recall the Poincar\'e metric expansion and the definition of the
renormalized volume  
coefficients $v_k$.  References for this material are  
\cite{G}, \cite{GH}, and \cite{FG2}.  Let $g$ be a metric of signature   
$(p,q)$ on a manifold $M$ of dimension $n\geq 3$.  There are versions of
most of the statements in dimension 2, but this case is anomalous and 
our main interest is in higher 
dimensions, so for simplicity we assume $n\geq 3$.   
If $n$ is odd, there is a smooth 1-parameter family $h_r$, $0\leq r< 1$, of 
metrics on $M$ such that $h_0 =g$ and such  
that the metric $g_+ = r^{-2}\left(dr^2 + h_r\right)$ of signature
$(p+1,q)$ on $M\times (0,1)$ 
satisfies that $\Ric(g_+)+ng_+$ vanishes to infinite order at $r=0$.  The
Taylor expansion in $r$ of $h_r$ at $r=0$ can be chosen to be even in $r$,
in which case it is uniquely determined.  

For $n\geq 4$ even, the corresponding statement holds only to a finite 
order.  
We say that a tensor is $O(r^m)$ if all of  
its components relative to a frame smooth up to $r=0$ are $O(r^m)$.  
We use lower case Latin indices to label objects on $M$.  
When $n$ is even, $h_r$ can be chosen so that its Taylor expansion is even
in $r$ and such that 
$$
\Ric(g_+)+ng_+=O(r^{n-2}),\qquad
h^{ij}\left(\Ric(g_+)+ng_+\right)_{ij}=O(r^n).
$$  
In the second equation, $\left(\Ric(g_+)+ng_+\right)_{ij}$ denotes the
component with both indices in the $M$ factor.  In calculating the trace,
$h^{ij}$ can be taken to be either $(h_0)^{ij}=g^{ij}$ or $(h_r)^{ij}$.   
These conditions uniquely 
determine $h_r \mod O(r^{n})$ and also $\operatorname{tr}_{g}h_r$ $\mod
O(r^{n+2})$.  For $n\geq 4$ even, there is a conformally invariant
trace-free, divergence-free natural tensor $\cO_{ij}$, the ambient
obstruction tensor, which obstructs the existence of 
a formal power series solution for $g_+$ to the next order.  $\cO_{ij}$
depends on derivatives of $g$ of order up to $n$.  When $n=4$, $\cO_{ij}$
is the classical Bach tensor.  

The passage from $g$ to $g_+$ is conformally invariant in the sense that if 
$\gh=e^{2\om}g$ with $\om\in C^\infty(M)$, and $\widehat{h}_r$ denotes the
expansion  
determined by $\gh$, then the metrics $g_+$ and 
$\gh_+=r^{-2}\left(dr^2 + \widehat{h}_r\right)$ are isometric 
by a diffeomorphism restricting to the identity on $M\times \{0\}$, 
to infinite order
if $n$ is odd, and up to a term which is $O(r^{n-2})$ and the trace of   
whose tangential component is $O(r^n)$ if $n$ is even.  

There are two special families of conformal structures in even dimensions  
for which the obstruction tensor vanishes and for which it is possible to 
uniquely determine the expansion of the Poincar\'e metric to infinite order
in a conformally 
invariant way.  These are the locally conformally flat structures and the
conformal classes containing an Einstein metric.  In these cases, the
normalized expansion can be written explicitly and terminates at order
four: for all $n\geq 3$, one has
\begin{equation}\label{confflat}
(h_r)_{ij}=g_{ij}-P_{ij}r^2+\tfrac14 P_{ik}P_j{}^kr^4
\end{equation}
if $g$ is Einstein or locally conformally flat.   
See \cite{SS} and \S 7 of \cite{FG2}.    

The volume form of $g_+$ is
$$
dv_{g_+}=r^{-n-1}dv_{h_r}dr
=r^{-n-1}\left(\frac{\det h_r}{\det h_0}\right)^{1/2}dv_{g}dr.
$$
The renormalized volume coefficients are defined by the Taylor expansion:  
\begin{equation}\label{expansion}
\left(\frac{\det h_r}{\det h_0}\right)^{1/2}
\sim 1+\sum_{k=1}^\infty (-2)^{-k}v_kr^{2k}.  
\end{equation}
Thus $v_k$ is uniquely determined by $g=h_0$ for
all $k\geq 1$ if $n$ is odd, and for $1\leq k\leq n/2$ if $n$ is even.   
As will be discussed in  
more detail below, $v_1$ and $v_2$ are given by   
$$
v_1=J,\qquad 
v_2=\tfrac12\left(J^2-P_{ij}P^{ij}\right),
$$ 
where $J=R/2(n-1)=P_i{}^i$.  
If $g$ is Einstein or locally conformally flat, then 
the $v_k$ are determined by $g$ for all $k$ for both $n$ even and 
odd.  Proposition 1 of \cite{GJ} uses \eqref{confflat} to 
show that $v_k(g)=\sigma_k(g^{-1}P)$ for all  
$k\geq 1$ if $g$ is locally conformally flat.  The same argument shows that
this also holds if $g$ is Einstein.  In particular, $v_k(g)$ is constant for
Einstein metrics.  

The evenness of the Poincar\'e metric in $r$ suggests to introduce $r^2$ as
a new variable.  Set  
$\rho = -\frac12 r^2$ and $g_\rho = h_r$. 
Then the volume expansion \eqref{expansion} becomes \eqref{ambexpansion}.   

The 1-parameter family $g_\rho$ can be characterized directly in terms of
the expansion of the   
ambient metric $\gt$ associated to $g$, which is equivalent to 
the expansion of the Poincar\'e metric $g_+$.
Define $\gt$, a metric of signature $(p+1,q+1)$ on    
$\R_+\times M\times (-1/2,0]\ni (t,x,\rho)$, by    
\begin{equation}\label{ambmetric}
\gt = 2t\,dt\,d\rho +2\rho\,dt^2 +t^2g_\rho.
\end{equation}
The condition
$\Ric(g_+)+ng_+=0$ is equivalent to $\Ric(\gt)=0$.  Thus the expansion of
$g_\rho$ can be thought of as arising from formally solving $\Ric(\gt)=0$
to the appropriate order rather than $\Ric(g_+)+ng_+=0$.  An advantage of  
considering $\gt$ is that $\gt$ is smooth near $\rho
=0$, whereas $g_+$ is singular at $r=0$.  As we will see, this makes it
easier to pass objects constructed out of $\gt$ back to $M$.    

For each $k\geq 1$ satisfying also $k<n/2$ if $n$ is even, the Taylor
coefficient $\pa_\rho^kg_\rho|_{\rho =0}$ is given by a polynomial natural 
tensor depending on the initial metric $g$.  For $n$ even, the trace   
$g^{ij}\left(\pa_\rho^{n/2}g_{ij}|_{\rho =0}\right)$ at order $n/2$ is also
a natural scalar invariant of $g$.  
It is possible to directly compute the beginning coefficients.  For  
example, 
letting $'=\pa_\rho$ and suppressing the ${}_\rho$ in $g_\rho$, (3.6) and
(3.18) of \cite{FG2} show that one has at $\rho =0$:   
\begin{equation}\label{2derivs}
g_{ij}'=2P_{ij},\qquad g_{ij}'' = \frac{2}{4-n}B_{ij} +2P_i{}^kP_{kj},  
\end{equation}
where 
$$
B_{ij} = P_{ij},_k{}^k-P_{ik},_j{}^k- P^{kl}W_{kijl}
$$
is the Bach tensor.  Here $W_{ijkl}$ denotes the Weyl tensor.  

The last part of \S 6 of \cite{FG2} considers a family of 
trace-free symmetric natural 2-tensors depending on a metric $g$.  
Here we call these the extended obstruction tensors.   
They have the feature that their    
transformation laws under conformal change are explicit and relatively
simple.  The first extended obstruction tensor is
$(4-n)^{-1}B_{ij}$, which we denote by $\Om^{(1)}_{ij}$.  Its   
well-known transformation law under conformal change $\gh=e^{2\om} g$ is:   
\begin{equation}\label{bachlaw}
e^{2\om}\wh{\Om}^{(1)}_{ij} = 
\Om^{(1)}_{ij}-2\om^kC_{(ij)k} +\om^k\om^l W_{kijl}, 
\end{equation}
where $C_{ijk}=P_{ij,k}-P_{ik,j}$ is the Cotton tensor.  Equation 
\eqref{2derivs} shows that $g_{ij}''|_{\rho=0}$ can be expressed in terms
of $\Om^{(1)}_{ij}$ by:  
\begin{equation}\label{firstone}
\tfrac12 g_{ij}''|_{\rho =0} = \Om^{(1)}_{ij} + P_i{}^kP_{kj}. 
\end{equation}
The definition and basic properties of the extended obstruction tensors    
are part of the theory of conformal curvature tensors developed       
in \S 6 of \cite{FG2}.  We summarize the relevant considerations and refer
to \cite{FG2} for details.   

Consider the curvature tensor and its covariant derivatives for
an ambient metric \eqref{ambmetric}.  We denote its curvature tensor 
by $\Rt$, with components $\Rt_{IJKL}$.  Here capital Latin indices are
used for objects on $\R_+\times M\times \R\ni (t,x,\rho)$, and we use '0'
for the $\R_+$ factor ($t$ component), lower case Latin indices for the $M$
factor, and '$\infty$' for the $\R$ factor ($\rho$ component).  For $r\geq
0$, the $r$-th covariant derivative $\widetilde{\nabla}^r\Rt$ of the 
curvature tensor of $\gt$ will be denoted   
$\Rt^{(r)}$, with components $\Rt^{(r)}_{IJKL,M_1\cdots M_r}$.
Sometimes the superscript ${}^{(r)}$ is omitted when the list of
indices makes clear the value of $r$.  

The conformal curvature tensors are tensors on $M$ obtained from
from the covariant derivatives of curvature of
$\gt$ as follows.  Choose an order $r\geq 0$ of covariant differentiation.  
Divide the set of symbols $IJKLM_1\cdots M_r$ into three disjoint
subsets labeled $\cS_0$, $\cS_M$ and $\cS_\nf$.  Set the indices in $\cS_0$
equal to $0$, those in $\cS_\nf$ equal to $\nf$, and let those in $\cS_M$
correspond to $M$ in the decomposition $\R_+ \times M\times \R$.  
Evaluate the resulting component $\Rt_{IJKL,M_1\cdots M_r}$ at $\rho = 0$
and $t=1$.  This defines a tensor on $M$, sometimes denoted by
$\Rt^{(r)}_{\cS_0,\cS_M,\cS_\nf}$, whose rank is the cardinality of the set
$\cS_M$.  In local coordinates, the indices in $\cS_M$ vary between $1$ and
$n$.  

The simplest case is $r=0$.  The curvature tensor $\Rt$ of $\gt$ satisfies 
$\Rt_{IJK0}=0$, so we must choose $\cS_0 = \emptyset$ in order 
to get a nonzero component.  Up to reordering the indices, there are only
three possible nonzero choices (see (6.2) of \cite{FG2}):    
\begin{equation}\label{curv0}
\Rt_{ijkl}|_{\rho = 0, t=1}=W_{ijkl},\quad
\Rt_{\nf jkl}|_{\rho = 0,\,t=1} = C_{jkl}\quad
\Rt_{\nf ij\nf}|_{\rho =0,\,t=1}=\frac{B_{ij}}{4-n}.   
\end{equation}
Thus the conformal  
curvature tensors which arise for $r=0$ are precisely the Weyl, Cotton, and
Bach tensors of $g$, except that when $n=4$, the Bach tensor arises as the
obstruction tensor rather than as a conformal curvature tensor.  

Since $g_\rho$ is uniquely determined by $g_0=g$ to infinite order for $n$
odd, it follows that for $n$ odd the conformal curvature tensors     
$\Rt^{(r)}_{\cS_0,\cS_M,\cS_\nf}$ for all choices of $r$ 
and $\cS_0,\cS_M,\cS_\nf$ are defined and are polynomial natural  
tensors.  However, when $n$ is even, one must restrict the orders of
differentiation to avoid the indeterminacy of $g_\rho$ at order $n/2$.    
For $n$ even, the tensor $\Rt^{(r)}_{\cS_0,\cS_M,\cS_\nf}$ depends only on 
$g$ and is a natural tensor so long as $s_M + 2s_\infty \leq n+1$, where 
$s_M$, $s_\infty$ are the cardinalites of $\cS_M$, $\cS_\nf$, resp.  
If $g$ is Einstein or locally conformally flat, then also for $n$ even the
conformal curvature tensors are defined for all choices of $r$ 
and $\cS_0,\cS_M,\cS_\nf$.  As will be seen below, they all 
vanish in the locally conformally flat case.   

Because $\gt$ changes by a diffeomorphism when $g$ is changed conformally, 
the covariant derivatives $\widetilde{\nabla}^r \Rt$ of
ambient curvature transform tensorially 
under conformal change of $g$.  This leads to an explicit  
identification of the transformation laws of the  
conformal curvature tensors under conformal change.  The following  
is Proposition 6.5 of \cite{FG2}.      
\begin{proposition}\label{ctransform}
Let $g$ and $\wh{g}=e^{2\om}g$ be conformally related metrics on $M$.  
Let $IJKLM_1\cdots M_r$ be a list of indices, $s_0$ of which are $0$,
$s_M$ of which correspond to $M$, and $s_\nf$ of which are $\nf$.  
If $n$ is even,
assume that $s_M + 2s_\nf \leq n+1$.  Then the conformal curvature tensors
satisfy the conformal transformation law:
\begin{equation}\label{transformula}
e^{2(s_\nf -1)\om}  \wh{\Rt}_{IJKL,M_1\cdots M_r}|_{\wh{\rho} = 0,\,\wh{t}=1}
= \Rt_{ABCD,F_1\cdots F_r}|_{\rho = 0,\,t=1}p^A{}_I
\cdots p^{F_r}{}_{M_r},
\end{equation}
where $p^A{}_I$ is the matrix
\begin{equation}\label{pmatrix}
p^A{}_I = 
\left(
\begin{matrix}
1&\om_i&-\frac12 \om_k\om^k\\
0&\delta^a{}_i&-\om^a\\
0&0&1
\end{matrix}
\right).
\end{equation}
\end{proposition}
\noindent
Here the conformal curvature tensor $\Rt_{IJKL,M_1\cdots M_r}|_{\rho =  
0,\,t=1}$ evaluated for the metric $\gh$ is denoted 
$\wh{\Rt}_{IJKL,M_1\cdots M_r}|_{\wh{\rho} = 0,\,\wh{t}=1}$.    
The variables $\wh{\rho}$ and $\wh{t}$ denote the    
coordinates on $\R_+\times M\times \R$, thought of as a separate copy from 
the space for the unhatted metric.  In \eqref{pmatrix}, indices on 
$\om_i$ are raised using the initial metric $g$.  

In expanding the 
right hand side of \eqref{transformula}, the leading term arises by
replacing each $p$ by $\delta$, giving
$\Rt_{IJKL,M_1\cdots M_r}|_{\rho =  0,\,t=1}$.  
Because of the upper-triangular form of the matrix $p^A{}_I$, the 
other 
terms on the right hand side all involve ``earlier'' conformal curvature
tensors in the sense that each '$i$' can be replaced only by $0$ and each 
$\nf$ only by
an '$i$' or a $0$.  It is clear that the conformal transformation law of 
a conformal 
curvature tensor involves only other conformal curvature tensors and first
derivatives of $\om$.  In case $r=0$, using \eqref{curv0} and the fact that
$\Rt_{IJK0}=0$, one sees that
\eqref{transformula} reproduces the  
conformal invariance of the Weyl tensor and the usual conformal
transformation laws of the Cotton and Bach tensors.  Equation
\eqref{transformula} can be interpreted as asserting that 
$\widetilde{\nabla}^r\Rt\,|_{\rho =  0}$ defines a section of the     
\mbox{$(r+4)$-th} tensor power of the cotractor bundle of the conformal   
manifold  $(M,[g])$ with a particular conformal weight.    

It follows directly from the definition that the conformal curvature
tensors all vanish if $g$ is flat.  Thus a consequence of 
\eqref{transformula} is that also they all vanish if $g$ is locally
conformally flat.  
By the infinite order invariance of
the ambient metric for $n$ even in the locally conformally flat case, this
is true for all conformal curvature tensors in both even and odd  
dimensions.   

We now define the extended obstruction tensors.
\begin{definition}\label{extobs}
Let $k\geq 1$.  Suppose that $n$ is odd or $n$ is even and $n>2(k+1)$.  
Define the $k$-th extended obstruction tensor 
$\Om^{(k)}_{ij}$ to be the conformal curvature tensor:
$$
\Om^{(k)}_{ij} = \Rt_{\infty  
  ij\infty,\,\underbrace{\scriptstyle{\infty\ldots\infty}}_{k-1}}|_{\rho 
  =0,\,t=1}.  
$$
\end{definition} 

According to the above discussion, $\Om^{(k)}_{ij}$ is a polynomial
natural tensor of the initial metric $g$.  
For $k=1$, \eqref{curv0} shows that $\Om^{(1)}_{ij}$ is given by
\eqref{omega1}.  It is clear that $\Om^{(k)}_{ij}$  
is symmetric in $ij$, and it is also trace-free:    
\begin{proposition}\label{tracefree}
For each $k\geq 1$ and in all dimensions $n$ as above for which
$\Om^{(k)}_{ij}$ is defined, one has
$$
g^{ij}\Om^{(k)}_{ij}=0.
$$
\end{proposition}
\begin{proof}
This is a consequence of the Ricci-flatness of the ambient metric (to the 
appropriate order if $n$ is even).  First suppose that $n$ is odd.  
Since $\Ric(\gt)=O(\rho^\nf)$, we have 
\begin{equation}\label{ricci0}
\gt^{IJ}\Rt_{KIJL,M_1\cdots M_r}=0
\end{equation}
at $\rho =0$ for all choices of $KLM_1\cdots M_r$.  Take all of
$KLM_1,\cdots M_r$ to be $\nf$.  At $\rho=0$ we have  
$$
\gt^{IJ} = 
\left(
\begin{matrix}
0&0&t^{-1}\\
0&t^{-2}g^{ij}&0\\
t^{-1}&0&0
\end{matrix}
\right).
$$
The terms in \eqref{ricci0} with $IJ=0\nf$ or $\nf 0$ vanish by 
skew-symmetry of $\Rt_{KIJL}$ in $KI$ and $JL$.  Thus \eqref{ricci0}
reduces to $g^{ij}\Rt_{\nf ij\nf,\nf\ldots \nf}=0$ as desired.  

The same argument applies if $n$ is even, so long as one checks that 
the order vanishing of $\Ric(\gt)$ is sufficient to under the
restriction $n>2(k+1)$.  This is precisely the statement of Proposition 6.4 
of \cite{FG2}.  
\end{proof}

Since for $g$ locally conformally flat, all conformal curvature tensors are 
defined and vanish whether $n$ is even or odd, in 
particular it follows that $\Om^{(k)}_{ij}$ is defined and 
$\Om^{(k)}_{ij}=0$ for all $k$ for locally conformally flat $g$.   
This is also true if $g$ is Einstein; see Proposition
7.6 of \cite{FG2}.  Note that general conformal curvature tensors do not  
vanish for Einstein metrics; for example the Weyl tensor is a conformal
curvature tensor.  For $n$ even and $g$ Einstein, the vanishing of
$\Om^{(n/2-1)}_{ij}$ is actually the condition used to normalize the
indeterminacy in the ambient metric; see Proposition 7.7 of \cite{FG2}.

Each obstruction tensor has divergence zero.  But 
this property does not extend to the extended obstruction tensors.  Already
this fails for $k=1$:  the divergence of the Bach tensor is given by
$B_{ij,}{}^j = (n-4)P^{jk}C_{jki}$.  

Next we define ``higher Cotton tensors'', which will enter into the
conformal transformation law of the extended obstruction tensors.  
\begin{definition}\label{gencotton}
Let $k\geq 1$.  Suppose that $n$ is odd or $n$ is even and  $n\geq
2(k+1)$.  Define the $k$-th higher Cotton tensor $C^{(k)}_{ijl}$ 
by: 
$$
C^{(k)}_{ijl} = 2\Rt_{\infty  
  (ij)l,\,\underbrace{\scriptstyle{\infty\ldots\infty}}_{k-1}}
  +\Rt_{\infty  
  ij\infty,\,\underbrace{\scriptstyle{l\infty\ldots\infty}}_{k-1}}
  +\Rt_{\infty  
  ij\infty,\,\underbrace{\scriptstyle{\infty l\infty\ldots\infty}}_{k-1}} 
  +\cdots 
  +\Rt_{\infty  
  ij\infty,\,\underbrace{\scriptstyle{\infty\ldots\infty l}}_{k-1}}.
$$
Here all $\Rt$ components are evaluated at $\rho=0$, $t=1$.  
\end{definition}

For $k$ and $n$ as in Definition~\ref{gencotton}, $C^{(k)}_{ijl}$ is a  
polynomial  
natural tensor of the initial metric $g$.  As for the extended obstruction
tensors, it is defined and vanishes for 
all $k$ for $g$ 
Einstein or locally conformally flat.  Equation \eqref{curv0} shows   
that 
$$
C^{(1)}_{ijl}=2C_{(ij)l}.
$$
The tensors $C_{ijl}$ and $C^{(1)}_{ijl}$ are equivalent; 
$C_{ijl}$ can be recovered from $C^{(1)}_{ijl}$ by   
$C_{ijl}=\frac23 C^{(1)}_{i[jl]}$.  
It is clear that $C^{(k)}_{ijl}$ is symmetric in $ij$, and it is also
trace-free in these indices:
\begin{proposition}\label{Ctracefree} 
For each $k\geq 1$ and in all dimensions $n$ as above for which
$C^{(k)}_{ijl}$ is defined, one has
$$
g^{ij}C^{(k)}_{ijl}=0. 
$$
\end{proposition}
\begin{proof}
The proof is similar to that of Proposition~\ref{tracefree}.  
Again assume first that $n$ is odd.  Take $K$, $L$, and all but one of
the $M_s$ to be $\nf$ in \eqref{ricci0} to deduce just as in the proof of
Proposition~\ref{tracefree} that 
$g^{ij}\Rt_{\infty  
ij\infty,\,\underbrace{\scriptstyle{\infty \ldots l
    \ldots\infty}}_{k-1}}=0$   
for any location of the index $l$ after the comma.  The same argument
applied to the first term on the right hand side in 
Definition~\ref{gencotton} shows that at $\rho =0$ and $t=1$ we have 
$$
g^{ij}\Rt_{\infty  
  ijl,\,\underbrace{\scriptstyle{\infty\ldots\infty}}_{k-1}}
+\Rt_{\infty  
  0\nf l,\,\underbrace{\scriptstyle{\infty\ldots\infty}}_{k-1}}=0.
$$
Now (1) of Proposition 6.1 of \cite{FG2} states that 
$$
\Rt_{IJK0,M_1\cdots M_r}
=-\sum_{s=1}^r\Rt_{IJKM_s,M_1\cdots \wh{M_s}\cdots M_r}
$$
at $t=1$.  Applying this along with the symmetries of $\Rt$   
shows that
$\Rt_{\infty  
  0\nf l,\,\underbrace{\scriptstyle{\infty\ldots\infty}}_{k-1}}=
\Rt_{\nf l\infty  
  0,\,\underbrace{\scriptstyle{\infty\ldots\infty}}_{k-1}}
=0$, and the
result follows.

Proposition 6.4 of \cite{FG2} shows that the same argument
applies if $n$ is even and $n\geq 2(k+1)$.  
\end{proof}

We remark that $g^{jl}C^{(k)}_{ijl}=0$ for $1\leq k\leq 3$, but
not for $k=4$.  Also, the symmetry $C^{(1)}_{(ijl)}=0$  
satisfied by $C^{(1)}_{ijl}=2C_{(ij)l}$ does not hold for 
$C^{(2)}_{ijl}$. 

A special case of Proposition~\ref{ctransform} is the conformal 
tranformation law for the extended obstruction tensors:  
\begin{proposition}\label{obslaw}
Let $k\geq 1$.  Let $n$ be odd or even with $n>2(k+1)$. 
Under a conformal change $\gh=e^{2\om}g$, the conformally transformed
extended obstruction tensor is given by:
$$
e^{2k\om}\wh{\Om}^{(k)}_{ij} = \Om^{(k)}_{ij} + 
\sum{}' \Rt_{ABCD,F_1\cdots F_{k-1}}|_{\rho =
  0,\,t=1}p^A{}_\nf p^B{}_i p^C{}_j p^D{}_\nf p^{F_1}{}_\nf
\cdots p^{F_{k-1}}{}_\nf, 
$$
where $p^A{}_I$ is given by \eqref{pmatrix} and $\sum{}'$ denotes the sum 
over all indices except for $ABCDF_1\cdots F_{k-1}=\nf ij\nf \nf \cdots
\nf$.    
\end{proposition}

Thus the conformal transformation law of the extended obstruction tensors
is given explicitly in terms of conformal curvature tensors and first
derivatives of the conformal factor.  For $k=1$, this reproduces
\eqref{bachlaw}.  By the upper-triangular form of
$p^A{}_I$, all of the conformal curvature tensors appearing in $\sum{}'$
with nonzero coefficient are defined if $n$ is even and $n\geq 2(k+1)$.   
Next we identify the terms in the
transformation law which are linear in $\nabla \om$.  
\begin{proposition}\label{linearlaw}
Let $k$, $n$ be as in Proposition~\ref{obslaw}.  Under conformal change
$\gh=e^{2\om}g$, we have:
$$
e^{2k\om}\wh{\Om}^{(k)}_{ij} = \Om^{(k)}_{ij} -\om^l C^{(k)}_{ijl}
+O(|\nabla \om|^2).
$$
\end{proposition}
\begin{proof}
For a term in $\sum{}'$ in Proposition~\ref{obslaw} to be linear in
$\nabla\om$, all $p$'s but one must be $\delta$, and $p^0{}_\nf$   
terms are excluded.  
If we suppress writing $|_{\rho=0,\,t=1}$, we obtain:
\[
\begin{split}
e^{2k\om}\wh{\Om}^{(k)}_{ij} &- \Om^{(k)}_{ij}\\
=&-\om^l\left(\Rt_{l
  ij\nf,\,\underbrace{\scriptstyle{\infty\ldots\infty}}_{k-1}}
  +\Rt_{\nf
  ijl,\,\underbrace{\scriptstyle{\infty\ldots\infty}}_{k-1}}
  +\Rt_{\infty  
  ij\infty,\,\underbrace{\scriptstyle{l\infty\ldots\infty}}_{k-1}}
  +\cdots 
  +\Rt_{\infty  
  ij\infty,\,\underbrace{\scriptstyle{\infty\ldots\infty l}}_{k-1}}\right)\\
&+\om_i\Rt_{\infty  
  0j\infty,\,\underbrace{\scriptstyle{\infty\ldots\infty}}_{k-1}}
+\om_j\Rt_{\infty  
  i0\infty,\,\underbrace{\scriptstyle{\infty\ldots\infty}}_{k-1}}
+O(|\nabla\om|^2)\\
=&-\om^lC^{(k)}_{ijl} 
+\om_i\Rt_{\infty  
  0j\infty,\,\underbrace{\scriptstyle{\infty\ldots\infty}}_{k-1}}
+\om_j\Rt_{\infty  
  i0\infty,\,\underbrace{\scriptstyle{\infty\ldots\infty}}_{k-1}}
+O(|\nabla\om|^2).
\end{split}
\]
However, $\Rt_{\infty  
  0j\infty,\,\underbrace{\scriptstyle{\infty\ldots\infty}}_{k-1}}
=\Rt_{\infty  
  i0\infty,\,\underbrace{\scriptstyle{\infty\ldots\infty}}_{k-1}}
=0$ as in the proof of Proposition~\ref{Ctracefree}, and the result 
  follows.  
\end{proof}

It is possible to view the dimension as a formal parameter and thus regard
each of  
the extended obstruction tensors as a natural tensor depending rationally
on $n$; see the discussion at the end of \S 6 of \cite{FG2} 
(where, however, $n$ is called $d$).    
The following result, which is Proposition 6.7 of \cite{FG2}, justifies the  
name ``extended obstruction tensor''.  
\begin{proposition}\label{residue}
Viewed as a natural tensor rational in the dimension $n$, $\Om^{(k)}_{ij}$
has a simple pole at $n=2(k+1)$ with residue given by 
$$
\operatorname{Res}_{n=2(k+1)} \Om^{(k)}_{ij} = (-1)^k 
\left[ 2^{k-1}(k-1)!\right] ^{-1}\cO_{ij},
$$
where $\cO_{ij}$ denotes the obstruction tensor in dimension $2(k+1)$.  
\end{proposition}

As noted above, in the transformation law in Proposition~\ref{obslaw}, all
of the 
conformal curvature tensors appearing in $\sum{}'$ with nonzero coefficient 
are regular at $n=2(k+1)$.  Therefore, formally taking the residue of this 
transformation law at $n=2(k+1)$ recovers the conformal invariance of 
the obstruction tensor in dimension $2(k+1)$.  Likewise, for $k>1$ we may
consider 
the behavior as $n\rightarrow 2l$ with $2\leq l\leq k$.    
It can be shown that $\Om^{(k)}_{ij}$ and all the conformal curvature
tensors appearing in its transformation law 
have at most simple poles at $n=2k$.
It is possible to justify the relation obtained by formally taking the 
residue at $n=2k$ in the transformation law for $\Om^{(k)}_{ij}$; this
gives the conformal        
transformation law of $\operatorname{Res}_{n=2k}\Om^{(k)}_{ij}$.  In
general, the order of the poles increases with $k-l$.  For example,     
$\Om^{(3)}_{ij}$ has a double pole at $n=4$, with leading coefficient a
nonzero 
multiple of $B_i{}^kB_{kj}$.  In this case, consideration of the
coefficient of $(n-4)^{-2}$ in the 
transformation law in Proposition~\ref{obslaw} recovers the conformal  
invariance of $B_i{}^kB_{kj}$ in dimension 4.  

Now we turn to the proof of Theorem~\ref{coefficientform}, which asserts
that the Taylor coefficients in the ambient metric expansion  
can be written in terms of the Schouten tensor and the extended obstruction
tensors by formulae universal in the dimension.  

\medskip\noindent
{\it Proof of Theorem~\ref{coefficientform}}. 
We prove by induction on $k$ a stronger statement holding not only at $\rho
=0$.  Consider a metric $\gt$ of the form \eqref{ambmetric}, where now
$g_\rho$ is any smooth 1-parameter family of metrics on $M$, i.e. we make
no assumption that $\gt$ is asymptotically Ricci-flat.  For $k\geq 1$,
define 
$$
\Lambda^{(k)}_{ij} = 
\Rt_{\infty  
  ij\infty,\,\underbrace{\scriptstyle{\infty\ldots\infty}}_{k-1}}|_{t=1},
$$
a family of symmetric 2-tensors on $M$ parametrized by $\rho$.  
We claim that for each $k\geq 1$, there is a linear combination 
$\cQ_{ij}^{(k)}$
of partial contractions with respect to $g_\rho^{-1}$ of $g_\rho'$ 
and the $\La^{(l)}$, $1\leq l\leq k-1$, whose coefficients  
are independent of $n$, such that the identity
\begin{equation}\label{inductionidentity}
\pa_\rho^kg_{ij}= \cQ^{(k)}_{ij}\left(g',
\La^{(1)},\ldots,\La^{(k-1)}\right)
\end{equation}
holds for all $\rho$.  Since for $\gt$ asymptotically Ricci-flat, we have 
$g'|_{\rho =0}=2P$ and $\La^{(l)}|_{\rho=0} = \Om^{(l)}$ (for $l<n/2-1$ if
$n$ is even), the first statement of 
Theorem~\ref{coefficientform} follows 
upon setting $\rho =0$.  

Case $k=1$ of \eqref{inductionidentity} is trivial taking
$\cQ^{(1)}_{ij}=g'_{ij}$.  For $k=2$, we use an explicit calculation of the
component $\Rt_{\nf ij\nf}$ of a metric \eqref{ambmetric}.  The
Christoffel symbols of $\gt$ can be written explicitly; see (3.16) of
\cite{FG2}.  From this it is straightforward to calculate the curvature
tensor of $\gt$; see (6.1) of \cite{FG2}.  One obtains in particular
\begin{equation}\label{firstR}
\Rt_{\nf ij\nf}|_{t=1}= 
\frac12\left(g_{ij}''-\frac12 g^{kl}g'_{ik}g'_{jl}\right).
\end{equation}
Thus
\begin{equation}\label{k=2}  
g_{ij}''=2\La^{(1)}_{ij}+\frac12 g^{kl}g'_{ik}g'_{jl},
\end{equation}
which is a relation of the form \eqref{inductionidentity} for $k=2$.  

We need a preliminary calculation before proceeding with the induction
argument.  The calculation of the covariant derivative in terms of
Christoffel symbols gives
\[
\begin{split}
\Rt_{\infty  
  ij\infty,\,\underbrace{\scriptstyle{\infty\ldots\infty}}_{k+1}}
=\pa_\rho \Rt_{\infty  
  ij\infty,\,\underbrace{\scriptstyle{\infty\ldots\infty}}_{k}}
&-\Gat_{\nf\nf}^A\Rt_{A
  ij\infty,\,\underbrace{\scriptstyle{\infty\ldots\infty}}_{k}}
-\Gat_{i\nf}^A\Rt_{\nf
  Aj\infty,\,\underbrace{\scriptstyle{\infty\ldots\infty}}_{k}}\\
&-\Gat_{j\nf}^A\Rt_{\nf
  iA\infty,\,\underbrace{\scriptstyle{\infty\ldots\infty}}_{k}}
-\Gat_{\nf\nf}^A\Rt_{\nf
  ijA,\,\underbrace{\scriptstyle{\infty\ldots\infty}}_{k}}\\
&-\Gat_{\nf\nf}^A\Rt_{\nf
  ij\infty,\,\underbrace{\scriptstyle{A\ldots\infty}}_{k}}
-\ldots
-\Gat_{\nf\nf}^A\Rt_{\nf
  ij\infty,\,\underbrace{\scriptstyle{\infty\ldots A}}_{k}}.
\end{split}
\]
Now (3.16) of \cite{FG2} shows that these Christoffel symbols are given by:  
$$
\Gat^A_{\nf\nf}=0\quad \text{for all}\quad A
$$
and
$$
\Gat^0_{i\nf}=0,\qquad \Gat^l_{i\nf}=\tfrac12 g^{lm}g'_{im},\qquad
\Gat^\nf_{i\nf}=0.
$$
Therefore 
\begin{equation}\label{inductR}
\Rt_{\infty  
  ij\infty,\,\underbrace{\scriptstyle{\infty\ldots\infty}}_{k+1}}
=\pa_\rho \Rt_{\infty  
  ij\infty,\,\underbrace{\scriptstyle{\infty\ldots\infty}}_{k}}
-\tfrac12 g^{lm}g'_{im}\Rt_{\nf
  lj\infty,\,\underbrace{\scriptstyle{\infty\ldots\infty}}_{k}}
-\tfrac12 g^{lm}g'_{jm}\Rt_{\nf
  il\infty,\,\underbrace{\scriptstyle{\infty\ldots\infty}}_{k}}.
\end{equation}
The $\rho$ derivative commutes with restriction to $t=1$, so this can be
written in terms of the $\La^{(k)}_{ij}$ as  
\begin{equation}\label{inductderiv}
\pa_\rho \La^{(k)}_{ij}{}=\La^{(k+1)}_{ij}{} 
+ g^{lm}g'_{m(i}\La^{(k)}_{j)l}.
\end{equation}

Now we prove that there is an identity of the form
\eqref{inductionidentity} by induction on $k\geq 2$.  Suppose that 
\eqref{inductionidentity} holds for $k$.  Differentiate this relation with
respect to $\rho$.  Each of the summands in 
$\cQ^{(k)}_{ij}$ is a product of factors $g^{-1}$, $g'$, and the
$\La^{(l)}$ for $1\leq l\leq k-1$.  The derivative of any such factor is
again a sum of products of the same form, except that also $\La^{(k)}$ can
appear.  In fact,  
$(g^{-1})'= -g^{-1}g'g^{-1}$, $g''$ is given by \eqref{k=2}, and the 
derivative of a $\La^{(l)}$ by \eqref{inductderiv}.  Therefore the Leibnitz
rule gives a relation of the form \eqref{inductionidentity} for $k+1$.    
This completes the induction and thus also the proof of the first statement
of Theorem~\ref{coefficientform}.    

It is easily seen by induction starting with \eqref{k=2} and using
\eqref{inductderiv} that for $k\geq 2$, $\cQ^{(k)}_{ij}$ has the form  
\begin{equation}\label{firstterm}
\cQ^{(k)}_{ij}= 2\La^{(k-1)}_{ij} 
+ \overline{\cQ}^{(k)}_{ij}\left(g', \La^{(1)},\ldots,\La^{(k-2)}\right), 
\end{equation}
where $\overline{\cQ}^{(k)}_{ij}$ is a linear combination of partial
contractions of the indicated tensors.  Thus 
$$
\cG^{(k)}_{ij}= 2\Om^{(k-1)}_{ij} 
+ \overline{\cG}^{(k)}_{ij}\left(P, \Om^{(1)},\ldots,\Om^{(k-2)}\right)
$$
for some $\overline{\cG}^{(k)}_{ij}$.  It follows that 
\begin{equation}\label{K}
g^{ij}\pa_\rho^kg_{ij}|_{\rho=0} = g^{ij}\overline{\cG}^{(k)}_{ij}\left(P,    
\Om^{(1)},\ldots,\Om^{(k-2)}\right)
\end{equation}
for all $k\geq 2$ if $n$ is odd and for 
$2\leq k\leq n/2-1$ if $n$ is even.  However, this reasoning does not apply   
for $k=n/2$ if $n$ is even, since $\Om^{(n/2-1)}_{ij}$ is not defined.
Nonetheless we claim that this is true also for $k=n/2$, so that  
\begin{equation}\label{realK}
\cT^{(k)}=g^{ij}\overline{\cG}^{(k)}_{ij}
\end{equation}
in Theorem~\ref{coefficientform}.  To see this, the discussion following
(3.16) of \cite{FG2} shows  
that for $n$ even, $g^{ij}\pa_\rho^{n/2}g_{ij}|_{\rho=0}$ is determined by
the condition $\Rt_{\infty\infty}=O(\rho^{n/2-1})$.  Now
$$
\Rt_{\infty\infty,\,\underbrace{\scriptstyle{\infty\ldots\infty}}_{n/2-2}}  
=-\gt^{IJ}\Rt_{\infty
  IJ\infty,\,\underbrace{\scriptstyle{\infty\ldots\infty}}_{n/2-2}} 
=-t^{-2}g^{ij} \Rt_{\infty 
ij\infty,\,\underbrace{\scriptstyle{\infty\ldots\infty}}_{n/2-2}}.   
$$
Therefore $g^{ij} \Rt_{\infty ij\infty,
\,\underbrace{\scriptstyle{\infty\ldots\infty}}_{n/2-2}}|_{\rho=0}=0$  
if $\Rt_{\infty\infty}=O(\rho^{n/2-1})$.  Hence 
$g^{ij}\pa_\rho^{n/2}g_{ij}|_{\rho=0}$ is determined by requiring
$g^{ij} \La^{(n/2-1)}_{ij}|_{\rho =0} =0$.  
So setting $k=n/2$ in 
\eqref{firstterm}, taking the trace, and restricting to $\rho=0$
proves the second statement of Theorem~\ref{coefficientform} with 
$\cT^{(k)}$ given by \eqref{realK}.   
\stopthm

Equations \eqref{2derivs} and \eqref{firstone} show that     
$$
\cG^{(1)}_{ij}=2P_{ij},\qquad 
\cG^{(2)}_{ij}=2\Om^{(1)}_{ij} +2P_i{}^kP_{kj}. 
$$
Thus $\cT^{(2)}=2P^{ij}P_{ij}$.  Formulae for $\cG^{(k)}_{ij}$ 
for $k=3$, 4, 5 are given in \eqref{gformulae}.   

\medskip
\noindent
{\it Proof of Corollary~\ref{vform}}. 
Taylor expanding the square root of the determinant in \eqref{ambexpansion}
shows that 
$v_k$ can be written as a linear combination of complete contractions
of the Taylor coefficients $\pa_\rho^lg_{ij}|_{\rho=0}$ for $1\leq l\leq
k-1$ and also $g^{ij}\pa_\rho^kg_{ij}|_{\rho=0}$.  (See the end of this 
section for more details.)  Equation 
\eqref{derivform} shows that $\pa_\rho^lg_{ij}|_{\rho=0}$ for $1\leq l\leq 
k-1$ involves only the $\Om^{(s)}_{ij}$ with $s\leq k-2$, and  
\eqref{K} shows that this is also the case for  
$g^{ij}\pa_\rho^kg_{ij}|_{\rho=0}$.  
\stopthm

Theorem~\ref{atmost2} implies that for a fixed background metric $g$, the
equation $v_k(e^{2\om}g)=c$ is second order 
in the unknown $\om$, even though for $k\geq 2$, $v_k(g)$
depends on derivatives of $g$ of order up to $2k-2$.  
It is
possible to say more about the form of $v_k(e^{2\om}g)$ as a function 
of $\om$ with $g$ fixed.  
First we show that 
$\cG^{(k)}_{ij}$ and $\cV_k$ 
have a weighted homogeneity with respect their arguments.  
Consider a constant rescaling 
$\gh=s^2 g$ with $0<s\in \R$.  The ambient metrics 
\eqref{ambmetric} are related by the diffeomorphism 
$$
\widehat{t}=ts^{-1},\qquad \widehat{x}=x,\qquad
\widehat{\rho}=\rho s^2,
$$ 
with $\gh_{\widehat{\rho}}=s^2 g_{\rho}$.  It follows that   
$\pa_{\widehat{\rho}}^k\gh_{ij}|_{\widehat{\rho}=0}
=s^{2-2k}\pa_\rho^kg_{ij}|_{\rho=0}$.  Thus if $\widehat{\cG}^{(k)}_{ij}$
denotes $\cG^{(k)}_{ij}$ evaluated for the metric $\gh$, then 
\begin{equation}\label{Pm}
\widehat{\cG}^{(k)}_{ij}= s^{2-2k}\cG^{(k)}_{ij}.
\end{equation}
Suppose a term appears in
$\cG^{(k)}_{ij}$ whose homogeneity degrees with respect to 
$P$, $\Om^{(1)}, \ldots, \Om^{(k-1)}$ are  
$d_0$, $d_1, \ldots, d_{k-1}$, resp., and let 
$d=\sum_{l=0}^{k-1}d_l$ denote the total degree.  
Such a term necessarily involves $d-1$ contractions with respect to 
$g^{-1}$.  
By Proposition~\ref{obslaw}, the extended obstruction tensors transform by 
$\widehat{\Om}^{(l)}=s^{-2l}\Om^{(l)}$, and of course 
$\widehat{P}=P$ and $\gh^{-1}=s^{-2}g^{-1}$.
Thus \eqref{Pm} gives
$$
-2(d-1)+\sum_{l=1}^{k-1}(-2l)d_l = 2-2k,
$$ 
or
\begin{equation}\label{homo}
\sum_{l=0}^{k-1}(l+1)d_l = k.
\end{equation}
This same relation holds for terms appearing in $\cV_k$ 
since $\widehat{v}_k=s^{-2k}v_k$ and $\cV_k$ involves one more  
contraction because it is a scalar.  Of course, $d_{k-1}=0$ for $\cV_k$.  

\medskip
\noindent
{\it Proof of Theorem~\ref{structure}}.
Write $v_k(g)=\cV_k\left(P, \Om^{(1)},\ldots,\Om^{(k-2)}\right)$ as
a linear combination of complete contractions of $P$
and the $\Om^{(l)}$ as in Corollary~\ref{vform}.  The contractions which
occur all satisfy \eqref{homo} with $d_{k-1}=0$.  Collect the contractions 
according to their homogeneity degree $m(=d_0)$ in $P$:  write  
$$
v_k(g)=\sum_{m=0}^k\cV_{k,m}\left(P,
\Om^{(1)},\ldots,\Om^{(k-2)}\right), 
$$
where $\cV_{k,m}$ is the sum of the contractions which are
homogeneous of degree $m$ in $P$.  Observe first that 
$\cV_{k,k-1}=0$ since there are no solutions to \eqref{homo} with 
$d_0=k-1$.  Next, note that $\cV_{k,k}$ depends only on $P$ since $d_0=k$ 
in \eqref{homo} forces $d_l=0$ for $l>0$.  Also, 
$v_k(g)=\cV_{k,k}(P)$ if $g$ is conformally flat, since in this case 
all $\Om^{(l)}=0$.  Since $v_k(g)=\sigma_k(g^{-1}P)$ for $g$ conformally   
flat, it follows that $\cV_{k,k}(P)=\sigma_k(g^{-1}P)$ for general 
$g$ because any symmetric 2-tensor $P_{ij}$ at a point arises as the  
Schouten tensor of some conformally flat metric.  Thus 
$$
v_k(g)=\sigma_k(g^{-1}P)+\sum_{m=0}^{k-2}\cV_{k,m}\left(P,  
\Om^{(1)},\ldots,\Om^{(k-2)}\right),
$$
where $\cV_{k,m}\left(P,\Om^{(1)},\ldots,\Om^{(k-2)}\right)$ is  
homogeneous of degree $m$ in $P$.  

Evaluating at $\wh{g}$ gives 
$$
v_k(\gh)=\sigma_k(\gh^{-1}\wh{P})+\sum_{m=0}^{k-2}\wh{\cV}_{k,m}, 
$$
where $\wh{\cV}_{k,m}$ denotes
$\cV_{k,m}\left(P,\Om^{(1)},\ldots,\Om^{(k-2)}\right)$ evaluated for the 
metric $\gh$, i.e. $P$ and the $\Om^{(l)}$ are replaced by 
$\wh{P}$, $\wh{\Om}^{(l)}$, and the contractions are taken with respect to 
$\gh$.  Now $\sigma_k(\gh^{-1}\wh{P})=e^{-2k\om}\sigma_k(g^{-1}\wh{P})$.  
If we take into account the scaling of $v_k$ and of the  
$\Om^{(l)}$ as in the proof of \eqref{homo}, it follows that  
$$
\wh{\cV}_{k,m}=e^{-2k\om}\cV_{k,m}
\left(\wh{P},e^{2\om}\wh{\Om}^{(1)},
\ldots,e^{2(k-2)\om}\wh{\Om}^{(k-2)}\right), 
$$
where on the right hand side, 
$\cV_{k,m}
\left(\wh{P},e^{2\om}\wh{\Om}^{(1)},
\ldots,e^{2(k-2)\om}\wh{\Om}^{(k-2)}\right)$ denotes the sum of the 
contractions with respect to $g$ of the indicated tensors.  
Each of the $e^{2l\om}\wh{\Om}^{(l)}$ is given by Proposition~\ref{obslaw}, 
so is a polynomial in $\nabla\om$ with coefficients depending on $g$.    
So if we set 
$$
r_{k,m}(x,\nabla\om,\widehat{P})=
\cV_{k,m}\left(\wh{P},e^{2\om}\wh{\Om}^{(1)},
\ldots,e^{2(k-2)\om}\wh{\Om}^{(k-2)}\right),
$$ 
where the $x, \nabla\om$ arguments in $r_{k,m}$ correspond to the 
$e^{2\om}\wh{\Om}^{(1)},\ldots,e^{2(k-2)\om}\wh{\Om}^{(k-2)}$ arguments in 
$\cV_{k,m}$ and the $\wh{P}$ arguments correspond on both sides, 
then $r_{k,m}$ is a
polynomial in $(\nabla\om$, $\wh{P})$ homogeneous of degree $m$ in
$\wh{P}$, with coefficients depending on $g$.  
It remains only to bound its degree in $\nabla\om$.  

Consider the expression of $e^{2l\om}\wh{\Om}^{(l)}_{ij}$ given by
Proposition~\ref{obslaw}.  Set $\|0\|=0$, $\|i\|=1$ for $1\leq i\leq n$, 
$\|\infty\|=2$, and 
$\|AB\cdots C\|=\|A\|+\|B\|+\cdots +\|C\|$.  Now $p^A{}_I=0$ if 
$\|A\|>\|I\|$ and $p^A{}_I$ is homogeneous of degree $\|I\|-\|A\|$ in 
$\nabla\om$ for $\|I\|\geq \|A\|$.  So the term 
$$
\Rt_{ABCD,F_1\cdots F_{l-1}}|_{\rho =
  0,\,t=1}p^A{}_\nf p^B{}_i p^C{}_j p^D{}_\nf p^{F_1}{}_\nf
\cdots p^{F_{l-1}}{}_\nf 
$$
in Proposition~\ref{obslaw} is of degree  
$\leq 2l+4-\|ABCDF_1\cdots F_{l-1}\|$ in $\nabla\om$.  The conformal
curvature tensors have the property that 
$\Rt_{ABCD,F_1\cdots F_{l-1}}|_{\rho = 0,\,t=1}=0$ if
$\|ABCDF_1\cdots F_{l-1}\|\leq 3$.  This is because in this case at most
three of the indices $ABCDF_1\cdots F_{l-1}$ are nonzero; see Proposition
6.1 of \cite{FG2}.  Thus $e^{2l\om}\wh{\Om}^{(l)}_{ij}$ is of degree 
$\leq 2l$ in $\nabla\om$.  If $d_1, \cdots, d_{k-2}$ denote the homogeneity
degrees with respect to $\Om^{(1)}_{ij}, \cdots, \Om^{(k-2)}_{ij}$,
resp., of a contraction appearing in 
$\cV_{k,m}$, it follows that 
$r_{k,m}(x,\nabla\om,\widehat{P})$ has degree in $\nabla\om$ at most 
$$
\sum_{l=1}^{k-2}2ld_l=2(k-m-\sum_{l=1}^{k-2}d_l).
$$  
Since $d_0=m<k$, \eqref{homo} shows that $d_l>0$ for at least one 
$l\geq 1$, giving the upper bound $2k-2m-2$ for the degree of 
$r_{k,m}$, as  
claimed.  Clearly, for a specific contraction this argument gives a
possibly better bound depending on $\sum_{l=1}^{k-2} d_l$.     
\stopthm

It is possible to derive by hand formulae for some of the extended   
obstruction tensors and expressions for ambient metric coefficients
and renormalized volume coefficients in terms of them.  Consider first the
extended obstruction tensors.  We have already seen that $\Om^{(1)}_{ij}$
is given by \eqref{omega1}.  Formulae for higher extended obstruction
tensors in terms of the Taylor coefficients $\pa_\rho^kg_\rho|_{\rho =0}$
of the ambient metric may be derived inductively starting with
\eqref{firstR} and using \eqref{inductR}.  
For instance, \eqref{firstR} together with \eqref{inductR} for $k=0$ give: 
$$
\Rt_{\infty ij\infty,\infty}=
\tfrac12 g_{ij}'''-\tfrac12 g_{k(i}''g'_{j)}{}^k
+\tfrac14 g'^{kl}g_{ik}'g_{jl}' -g'_{k(i}\Rt_{j)\infty\infty}{}^k.   
$$
$g'$ and $g''$ at $\rho=0$ are given by \eqref{2derivs} and 
$g'''|_{\rho =0}$ in (3.18) of \cite{FG2}.  Substituting these gives 
\[
\begin{split}
(n&-4)(n-6)\Om^{(2)}_{ij}=
B_{ij,k}{}^k -2W_{kijl}B^{kl}-4P_k{}^kB_{ij}\\
&+(n-4)\left(4P^{kl}C_{(ij)k,l}
-2C^k{}_i{}^lC_{ljk}
+C_i{}^{kl}C_{jkl}+2P^k{}_{k,l}C_{(ij)}{}^l-2W_{kijl}P^k{}_mP^{ml}\right). 
\end{split}
\]
Carrying out the algorithm by hand to derive the formulae for a few more 
extended obstruction tensors in terms of the 
$\pa^k_\rho g_{ij}|_{\rho =0}$ is  
manageable; this uses only the form \eqref{ambmetric} of the ambient
metric.  But deriving formulae for $\pa^k_\rho g_{ij}|_{\rho =0}$ 
in terms of the curvature of the base metric for $k\geq 4$ by solving
the Einstein equation is more lengthy.     

A similar calculation gives the second Cotton tensor
$$
C^{(2)}_{ijl}=\left(2\Rt_{\infty(ij)l,\infty}+\Rt_{\infty ij\infty,l}
\right)\Big{|}_{\rho=0, t=1}.
$$
The Bianchi identity allows this to be rewritten as
\begin{equation}\label{cotton2}
C^{(2)}_{ijl}=\left(3\Rt_{\infty ij\infty,l}- \Rt_{\infty li\infty,j}
- \Rt_{\infty lj\infty,i}\right)\Big{|}_{\rho=0, t=1}.
\end{equation}
The covariant derivative can be evaluated using 
\eqref{curv0} and the formulae for the Christoffel symbols of $\gt$ given
by (3.16) of \cite{FG2} to obtain 
$$
\Rt_{\infty ij\infty,l}|_{\rho=0, t=1}
=\frac{B_{ij,l}}{4-n}-2P_l{}^mC_{(ij)m}.
$$
Substituting this into \eqref{cotton2} gives the desired formula for 
$C^{(2)}_{ijl}$.  

The proof of Theorem~\ref{coefficientform} gives the algorithm to make
explicit the formulae \eqref{derivform} for the ambient 
metric coefficients in terms 
of the Schouten tensor and the extended obstruction tensors.  This involves
the same ingredients as in the derivation of the formulae for the extended 
obstruction tensors discussed above; it is just a matter of which set of
quantities one is solving for inductively in terms of which others.  Again,
these 
relations depend only on the form \eqref{ambmetric} of the ambient metric
and not on the values of its Taylor coefficients obtained by solving the 
Einstein equation for $\gt$.  

Set $g^{(k)}_{ij}=\pa^k_{\rho}g_{ij}|_{\rho =0}$.  We have already seen
that 

\smallskip
${}\qquad\frac12 g^{(1)}_{ij}=P_{ij}$

\smallskip
${}\qquad\frac12 g^{(2)}_{ij} = \Om^{(1)}_{ij} + P_i{}^kP_{jk}$.   

\medskip
\noindent
Carrying out the algorithm of the
proof of Theorem~\ref{coefficientform}, one obtains:
\begin{equation}\label{gformulae}
\begin{split}
\tfrac12 g^{(3)}_{ij} =& \Om^{(2)}_{ij} + 4P^k{}_{(i}\Om^{(1)}_{j)k}\\ 
\tfrac12 g^{(4)}_{ij} =& \Om^{(3)}_{ij} + 6P^k{}_{(i}\Om^{(2)}_{j)k}
+4\Om^{(1)}{}^k{}_i\Om^{(1)}_{jk} +4P^k{}_iP^l{}_j\Om^{(1)}_{kl}\\
\tfrac12 g^{(5)}_{ij} =& \Om^{(4)}_{ij} + 8P^k{}_{(i}\Om^{(3)}_{j)k}
+14\,\Om^{(2)}{}^k{}_{(i}\Om^{(1)}_{j)k} +10P^k{}_iP^l{}_j\Om^{(2)}_{kl}
+16\,P^k{}_{(i}\Om^{(1)l}_{j)}\Om^{(1)}_{kl}.
\end{split}
\end{equation}

\noindent
The algorithm also easily gives the leading terms:   
for $k\geq 3$ one has     
$$
\tfrac12 g^{(k)}_{ij}=  
\Om^{(k-1)}_{ij}+2(k-1)P^l{}_{(i}\Om^{(k-2)}_{j)l} 
+\cJ^{(k)}_{ij},
$$
where $\cJ^{(k)}_{ij}$ is a linear combination of  
contractions of the tensors $P, 
\Om^{(1)},\ldots,\Om^{(k-3)}$ satisfying \eqref{homo}.  

Finally, the renormalized volume coefficients can be calculated from the
ambient metric coefficients by expanding the volume form.  Set 
$D=\det g_\rho/\det g_0$.
Then $D'=Dg^{ij}g'_{ij}$.  Successive differentiation of this relation
gives formulae for $\pa^k_\rho D/D$ in terms of $g^{-1}$ and derivatives
of $g$.  For example,
$$
D''=Dg^{ij}g''_{ij}-Dg^{ik}g^{jl}g'_{kl}g'_{ij}+D(g^{ij}g'_{ij})^2.
$$
The Taylor coefficients of $D$ are then obtained by evaluating these 
relations at $\rho=0$ and substituting the above formulae for the Taylor
coefficients of $g_{\rho}$.  Composing with the Taylor expansion of
$\sqrt{x}$ about $x=1$ gives the $v_k$ according to
\eqref{ambexpansion}.  It is 
straightforward but tedious to carry this out.  The result for the first
few $v_k$ is:
\begin{equation}\label{vkformulae}
\begin{split}
v_1 =&\sigma_1\\
v_2 =& \sigma_2\\
v_3 =& \sigma_3 - \tfrac13 \tr \left(P\Om^{(1)}\right)\\
v_4 =& \sigma_4 
+\tfrac13 \tr \left(P^2\Om^{(1)}\right)
-\tfrac13 (\tr P) \tr \left(P\Om^{(1)}\right)
-\tfrac{1}{12}\tr \left(P\Om^{(2)}\right) 
-\tfrac{1}{12}\tr \left(\Om^{(1)}\right)^2.
\end{split}
\end{equation}
Here we have omitted the argument $g^{-1}P$ of the $\sigma_k$.  Also  
omitted are the $g^{-1}$ factors raising the indices in the trace terms.   
These $\sigma_k$ are given by:  
\[
\begin{split}
\sigma_1=&J\\
\sigma_2=&\tfrac12 \left(J^2-\tr P^2\right)\\
\sigma_3=&\tfrac16\left(2\tr P^3 -3J\tr P^2 +J^3\right)\\
\sigma_4 =&
\tfrac{1}{24}\left(-6\tr P^4 +8J\tr P^3 +3(\tr P^2)^2
-6J^2\tr P^2 +J^4\right),
\end{split}
\]
where $J=\tr P=R/2(n-1)$.

\section{Linearization}\label{lin}

Let $X$ be a manifold-with-boundary and set $\partial X=M$.  If $[g]$ is a
conformal class of metrics of signature $(p,q)$ on $M$, recall that a
metric $g_+$ of signature $(p+1,q)$ on $X^\circ$ is said to be conformally
compact with conformal infinity $(M,[g])$ if $u^2g_+$ extends smoothly to 
$X$ with $u^2g_+|_M$ nondegenerate and $u^2g_+|_{TM}\in [g]$, where $u$ is
a defining function for $M$.   
The function $|du|^2_{u^2g_+}\big{|}_M$ is independent of the choice of
$u$; 
$g_+$ is said to be asymptotically hyperbolic if
$|du|^2_{u^2g_+}\big{|}_M=1$.  

Let $g_+$ be asymptotically hyperbolic and let $g$ be a choice of metric in
the conformal class on $M$.  Then there 
is an open neighborhood of $M$ ($= M\times \{0\}$) in $M\times
[0,\infty)$ on which 
there is a unique diffeomorphism $\varphi$ to a neighborhood of $M$ in
$X$ such that $\p|_M$ is 
the identity, and such that $\p^*g_+$ takes the form
$$
\p^*g_+ = r^{-2}\left(dr^2+h(r)\right),
$$
where $h(r)$ (denoted $h_r$ previously) is a 1-parameter family of metrics
on $M$ of signature $(p,q)$  
satisfying $h(0)=g$.  Here $r$ denotes the variable in $[0,\infty)$.  
See \S 5 of \cite{GL}.  

Suppose we choose a conformally related metric
$\gh = e^{2\om}g$, where $\om\in C^\infty(M)$.  
Then $\gh$ induces another diffeomorphism $\ph$ from a neighborhood of
$M$ in 
$M\times [0,\infty)$ to a neighborhood of $M$ in $X$ such that  
$\ph^*g_+ = r^{-2}\left(dr^2 +\widehat{h}(r)\right)$, where
$\widehat{h}(r)$ is a
1-parameter family of metrics on $M$, satisfying $\widehat{h}(0) =\gh$,
uniquely determined by $g_+$, $g$, and $\om$.  Consider     
the infinitesimal dependence of $\hh(r)$ on $\om$.  For each $t$, 
denote by     
$\hh^t(r)$ the 1-parameter family of metrics obtained from the conformal 
representative $\gh^t = e^{2t\om}g$.  Let 
$\delta = \pa_t|_{t=0}$ denote the operation of taking the infinitesimal
conformal variation.  For example, 
$$
\delta h(r) = \pa_t \hh^t(r)|_{t=0}. 
$$
We sometimes suppress 
writing the argument for $h(r)$ and $\delta h(r)$; the $r$ dependence of
$h$ and $\delta h$ is to be understood.  
\begin{theorem}\label{conflaw}
Under infinitesimal conformal change of $g$, $h(r)$ transforms by:    
\begin{equation}\label{law}
(\delta h)_{ij} = \om( 2 -r\pa_r)h_{ij} +2 \nabla_{(i}X_{j)},    
\end{equation}
where $X^i$ is the $r$-dependent family of vector fields on $M$ given by 
\begin{equation}\label{Xformula}
X^i(r)=\int_0^rsh^{ij}(s)\,ds\; \pa_j\om.
\end{equation}
Here $X_j(r)=h_{ij}(r)X^i(r)$, and 
$\nabla_i$ denotes the covariant derivative on $M$ with respect to  
$h(r)$ with $r$ fixed.
\end{theorem}
\noindent
Note in \eqref{Xformula} that $\pa_j\om$ is independent of $r$.
An immediate consequence of Theorem~\ref{conflaw} 
is the fact that for each $r$, the transformation rule for 
infinitesimal conformal change of $h(r)$ involves at most second  
derivatives of $\om$.      
\begin{proof}
For each $t$, we have a diffeomorphism $\varphi_t$ such that
$\varphi_t|_M$ is the identity and  
$\varphi_t^*g_+ = r^{-2}\left(dr^2 +\hh^t(r)\right)$.  So
$$
\left(\varphi_0^{-1}\circ \varphi_t\right)^*
\left(\frac{dr^2 +h(r)}{r^2}\right) 
=\frac{dr^2 +\hh^t(r)}{r^2}.
$$
Differentiate with
respect to $t$ at $t=0$ to deduce that there is a vector field $X$ near $M$  
in $M\times [0,\infty)$ such that $X|_M = 0$ and  
$$
\cL_X\left(\frac{dr^2 +h}{r^2}\right) = \frac{\delta h}{r^2},
$$
where $\cL$ denotes the Lie derivative.  
Expanding the left hand side and then multiplying by $r^2$ gives 
\begin{equation}\label{lie}
-2r^{-1}X(r)\left(dr^2 +h\right) +\cL_X(dr^2) + \cL_X h = \delta h.
\end{equation}
Now write $X = X^0\pa_r + X^i\pa_i$.  Then
\[
\begin{split}
X(r)&=X^0\\
\cL_X(dr^2) &= 2dX^0\,dr = 2\pa_rX^0dr^2+ 2\pa_iX^0dx^idr\\
\cL_Xh &= \left(2\nabla_{(i}X_{j)} +X^0 \pa_rh_{ij}\right) dx^idx^j
+2h_{ij}\pa_rX^j drdx^i,
\end{split}
\]
where $X_j = h_{jk}X^k$ and $\nabla_i$ is the covariant derivative on $M$
with respect to $h(r)$ with $r$ fixed.  Substitute these into \eqref{lie}
and then equate the  
coefficients of $dr^2$, $drdx^i$ and $dx^idx^j$ on the two sides of
\eqref{lie}.  One obtains 
\begin{equation}\label{components}
\begin{split}
-2r^{-1}X^0 +2\pa_rX^0 &= 0\\
2\pa_iX^0 + 2h_{ij}\pa_rX^j &=0\\
-2r^{-1}X^0h_{ij}+2\nabla_{(i}X_{j)} +X^0\pa_rh_{ij} &=\delta h_{ij}.  
\end{split}
\end{equation}
The first equation shows that $X^0 = cr$ where $c$ is independent of $r$,
i.e. $c$ is just a function of $x\in M$.   Substitute this into the last 
equation and evaluate at $r=0$.  Recalling that $X=0$ at $r=0$ and 
$\delta h = 2\om h =2\om g$ at $r=0$, one obtains $c=-\om$.  So now we know 
$X^0 = -\om r$.  Substitute this into the second equation to obtain
$$
\pa_r X^i = rh^{ij}\pa_j \om.     
$$
Now integrate in $r$ to solve for $X^i$; $\pa_j \om$ is a constant with
respect to the    
integration.  Using the initial condition $X^i=0$ at $r=0$, one obtains 
$$
X^i = \int_0^r sh^{ij}(s)\,ds\; \pa_j\om.
$$
Substituting $X^0 = -\om r$ into the third line of \eqref{components} gives  
\eqref{law}. 
\end{proof}

It is useful to introduce the new variable $\rho = -\frac12 r^2$ as in   
\S\ref{eot} in the infinitesimal transformation law \eqref{law}.
Set $g(\rho)=h(r)$ (denoted $g_\rho$ previously) and  
$Y^i(\rho)=X^i(r)$.  Then \eqref{law}, \eqref{Xformula} become
\eqref{newlaw}, \eqref{Yformula}.  

Consider now the case where $g_+=r^{-2}\left(dr^2+h(r)\right)$ is a
Poincar\'e metric whose Taylor expansion (to the appropriate order for $n$ 
even) is determined 
along $M$ by the choice of an initial metric $g$ via the  
Einstein equation $\Ric(g_+)=-ng_+$.  The Taylor coefficients of $g(\rho)$  
are the natural tensors studied in \S\ref{eot}, so the Taylor expansion of
\eqref{newlaw} gives the infinitesimal transformation laws
of these tensors.  For example, \eqref{Yformula} shows that 
$\pa_\rho Y^i|_{\rho =0}=-g^{ij}\pa_j\om$, 
so differentiating \eqref{newlaw} once at 
$\rho=0$ and recalling \eqref{2derivs} recovers the infinitesimal
transformation law  $\delta P_{ij}=-\om_{ij}$ of the Schouten tensor. 
In general, in \eqref{newlaw} the term $2\om( 1 -\rho\pa_\rho)g_{ij}$
encodes the scaling  
of each coefficient and the term $2 \nabla_{(i}Y_{j)}$ carries the
dependence on derivatives of $\om$.  It follows that the infinitesimal    
transformation laws of all these natural tensors (subject to the usual 
truncation for $n$ even) involves at most second derivatives of $\om$.

An easy consequence of Theorem~\ref{conflaw} is a similar formula for the 
infinitesimal transformation laws of the renormalized volume  
coefficients.  First suppose that $g_+$ is asymptotically hyperbolic with
conformal infinity $(M,[g])$ but not necessarily asymptotically Einstein,
as in the setting of Theorem~\ref{conflaw}.  Define $v(\rho)$ by 
\eqref{vdef}.    
\begin{proposition}\label{deltav}
Under infinitesimal conformal change of $g(0)$, $v(\rho)$ transforms by:  
\begin{equation}\label{vlaw}
\begin{split}
\delta v &= -2\om \rho\pa_\rho v +v \nabla_{i}^{(\rho)}Y^i\\
&= -2\om \rho\pa_\rho v + \nabla_{i}^{(0)}\left(vY^i\right),
\end{split}
\end{equation}
where $Y^i$ is given by \eqref{Yformula} and $\nabla_i^{(\rho)}$ is the 
covariant derivative with respect to $g(\rho)$ with $\rho$ fixed.  
\end{proposition}
\begin{proof}
Under a conformal tranformation, $\delta\det g(0) = 2n\om\det g(0)$, so  
$$
\frac{\delta v}{v} = \delta \log v= 
\tfrac12 \left( g^{ij}\delta g_{ij} -2n\om\right).  
$$
Substitution of \eqref{newlaw} gives  
$$
\frac{\delta v}{v} =  -\om g^{ij}\rho\pa_\rho g_{ij} 
+ g^{ij}\nabla_i^{(\rho)}Y_j = -2\om \frac{\rho\pa_\rho v}{v}   
+\nabla_i^{(\rho)}Y^i.    
$$
This gives the first line of \eqref{vlaw}.  The second line follows since 
$v \nabla_{i}^{(\rho)}Y^i=\nabla_{i}^{(0)}\left(vY^i\right)$ for any vector
field $Y^i$.  
\end{proof}

\medskip
\noindent
{\it Proof of Theorem~\ref{L}}.  
Take $g_+$ in Proposition~\ref{deltav} 
to be an asymptotically Einstein metric whose Taylor 
expansion is determined by $g=g(0)$.  By \eqref{ambexpansion}, the 
coefficient of $\rho^k$ in $\delta v$ is $\delta v_k$.  So 
taking Taylor coefficients in \eqref{vlaw} and recalling 
\eqref{Yformula} gives  
\eqref{dvkform} with $L^{(k)}_{ij}$ given by the first equality of 
\eqref{Lformula}.  The second equality of \eqref{Lformula} follows 
upon expanding via the Leibnitz rule.  
\stopthm

The fact that the second term on the right hand side of \eqref{vlaw}    
is a divergence implies that it drops out when considering the
infinitesimal conformal change of the volume of $M$ relative to the metrics
$g(\rho)$.  Suppose $M$ is compact and set 
$$
V(\rho)\equiv\Vol_{g(\rho)}(M) =\int_Mv(\rho)\,dv_{g(0)}.
$$
Integration of \eqref{vlaw} gives for each $\rho$:   
\begin{equation}\label{deltavol}
\delta V =\int_M\delta v\, dv_{g(0)}
=-2\int_M \om \rho\pa_\rho v\, dv_{g(0)}.
\end{equation}
Taking Taylor coefficients 
in \eqref{deltavol} (or integrating \eqref{dvkform} over $M$), it follows 
that:  
\begin{proposition}\label{deltavk}
Suppose $k\geq 1$ with $k\leq n/2$ if $n$ is even, and suppose $M$ is
compact.  Then  
$$
\int_M \delta v_k\,dv_g=-2k\int_M v_k \om \,dv_g.
$$
\end{proposition}

Proposition~\ref{deltavk} is the main ingredient in the proof of the result
of Chang-Fang \cite{CF}.  Consider the functionals 
$$
\cF_k(g)=\int_Mv_k(g)\,dv_g
$$ 
as $g$ varies over a conformal class.  $\cF_k$ is defined 
for all $k\geq 1$ if $k$ is odd and for $1\leq k\leq n/2$ if $n$ is even.
It was shown in \cite{G} that if $n$ is even, then $\cF_{n/2}(g)$ is 
conformally invariant, i.e. is constant on each conformal class.  The
Chang-Fang theorem gives the constrained Euler-Lagrange equation for the
other values of $k$:
\begin{theorem}\label{cf}
Suppose $k\geq 1$ and $k<n/2$ if $n$ is even.  The Euler-Lagrange equation
for $\cF_k(g)$ as $g$ varies over a conformal 
class, subject to the constraint $\Vol_g(M)=1$, is $v_k(g)=c$. 
\end{theorem}
\begin{proof}
The constrained Euler-Lagrange equation for $\cF_k$ is obtained by
requiring that  
$\cF_k -\lambda \Vol (M)$ vanishes to first order in $\om$ under a
conformal change $\gh=e^{2\om}g$, where $\lambda$ is a Lagrange multiplier.
This is therefore the condition     
\begin{equation}\label{EL}
\delta\left(\cF_k-\lambda \Vol (M)\right)(g)=0\quad \mbox{for all } \om.  
\end{equation}
Proposition~\ref{deltavk} together with the fact that 
$\delta\, dv_g=n\om\,dv_g$ give
$$
\delta \cF_k(g)=\int_M\left(\delta v_k\, dv_g+v_k \delta\, dv_g\right)
=(n-2k)\int_Mv_k\om\, dv_g.
$$
If $n=2k$ we recover the conformal invariance of $\cF_{n/2}$.  Otherwise 
\eqref{EL} becomes 
$$
(n-2k)\int_Mv_k\om\,dv_g-n\lambda \int_M \om \,dv_g=0 
\quad \mbox{for all } \om,
$$
which gives $v_k(g)=n\lambda/(n-2k)=c$.  
\end{proof}

Thus if we fix a background metric $g$ in the conformal class, then the 
critical points of $\cF_k(e^{2\om} g)$ as a function of $\om$ are those
$\om$ 
for which $v_k(e^{2\om}g)=c$.  So we recover the second order fully  
nonlinear operator whose structure was studied in \S\ref{eot}.  
Its linearization at $\om =0$ is of course just $\delta v_k$, 
so Theorem~\ref{L} gives:
\begin{proposition}\label{linear}
Suppose $k\geq 1$ with $k\leq n/2$ if $n$ is even.  
Let $\cP_k(\om)$ denote the linearization at $\om =0$ of the operator 
$\om\rightarrow v_k(e^{2\om}g)$ with $g$ fixed.  Then
$$
\cP_k(\om) = \nabla_i\left(L^{ij}_{(k)}\nabla_j \om\right)-2kv_k\om,
$$
with $L^{ij}_{(k)}$ given by \eqref{Lformula}.  
\end{proposition}

In considering \eqref{Lformula}, recall that $g_{ij}(\rho)$ is the series
determined by 
$g=g(0)$ upon formally solving the Einstein equation $\Ric(\gt)=0$.  Hence 
identification of   
$\pa_\rho^{l-1} g^{ij}|_{\rho =0}$ requires knowledge
of the coefficients  
$\pa_\rho^m g_{ij}|_{\rho =0}$ in  
the ambient metric expansion as well as calculation of the 
Taylor coefficients of the inverse in terms of these.  In any case, it is
clear from Theorem~\ref{coefficientform} and Corollary~\ref{vform} that 
each $L^{ij}_{(k)}$ (with $k\leq n/2$ for $n$ even) is a natural tensor
which can be written as a linear 
combination of contractions of the Schouten tensor and the extended
obstruction tensors with coefficients independent of the dimension.  

For small $k$, it is possible to calculate $L^{ij}_{(k)}$ from 
\eqref{Lformula} using \eqref{gformulae} and \eqref{vkformulae}.
Alternately, one can simply linearize the explicit 
expressions \eqref{vkformulae}.  In the latter approach, one uses from 
Proposition~\ref{linear} that $\cP_k$ is determined once one knows its
principal part $L^{ij}_{(k)}$.  Thus it suffices to calculate the principal
part of the linearization from \eqref{vkformulae}.  So in linearizing   
\eqref{vkformulae}, one can 
ignore contributions from the $\Om^{(l)}$ and simply apply the
Leibnitz rule to $P_{ij}$ and use $\delta P_{ij}=-\om_{ij}$.  

Recall that the linearization of 
$\sigma_k$ can be expressed in terms of its polarization.  If $A^i{}_j(t)$ 
is a 
1-parameter family of endomorphisms of a vector space, then the relation
\begin{equation}\label{Tdef}
\sigma_k(A)^{\textstyle{\cdot}} = \tr\left(T_{(k-1)}(A)\Dot A\right) 
\end{equation}
defines an endomorphism-valued polynomial $T_{(k-1)}(A)$ homogeneous   
of degree $k-1$ in $A$.
Let $\sigb_k$ denote the symmetric $k$-linear  
form obtained by complete polarization of $\sigma_k$, i.e. 
$\sigb_k(A_1,\ldots,A_k)$ is linear in each $A_l$, symmetric, and satisfies 
$\sigb_k(A,\ldots,A)=\sigma_k(A)$.  Then the Leibnitz rule shows that  
$$
\tr\left(T_{(k-1)}(A)B\right)=k\sigb_k(A,\ldots,A,B).    
$$
In the sequel, our vector space will be equipped with a non-degenerate
quadratic form 
which we use to raise and lower indices, and $A_{ij}$ will be symmetric.
So we will usually write \eqref{Tdef} in the form
$$
\sigma_k(A)^{\textstyle{\cdot}} = T^{ij}_{(k-1)}(A)\Dot A_{ij}.  
$$

The homogeneity of $\sigma_k$ together with $\delta P_{ij}=-\om_{ij}$ give 
\begin{equation}\label{deltasigmak}
\delta\left(
\sigma_k(g^{-1}P)\right)
=-T^{ij}_{(k-1)}(g^{-1}P)\om_{ij}-2k\sigma_k(g^{-1}P)\om,   
\end{equation}
so that the principal part of the linearization of $\sigma_k(g^{-1}P)$ is 
$-T^{ij}_{(k-1)}(g^{-1}P)$.  
In the following, we suppress writing the argument $g^{-1}P$ of
$T^{ij}_{(k-1)}$ and we write $\Om_{(l)}$ instead of $\Om^{(l)}$.   
Either directly linearizing \eqref{vkformulae} or calculating from
\eqref{Lformula}, one obtains:
\[
\begin{split}
L^{ij}_{(1)}&=-g^{ij}\\
L^{ij}_{(2)}&=-T^{ij}_{(1)}\\  
L^{ij}_{(3)}&=-T^{ij}_{(2)}+\tfrac13 \Om_{(1)}^{ij}\\  
L^{ij}_{(4)}&=-T^{ij}_{(3)}-\tfrac23 P_k{}^{(i}\Om_{(1)}^{j)k} 
+\tfrac13 P_{kl}\Om_{(1)}^{kl}g^{ij}+\tfrac13 P_k{}^k\Om_{(1)}^{ij} 
+\tfrac{1}{12}\Om_{(2)}^{ij}.
\end{split}
\]

Recall from the discussion in \S\ref{eot} that if $g$ is locally 
conformally flat, then the $v_k$ are defined for all $k$ also for $n$
even, and $v_k(g)=\sigma_k(g^{-1}P)$ for all $k$ in all dimensions.  The 
invariance of the ambient metric holds to all orders in all dimensions and
this was the fundamental ingredient used in the proof of
Theorem~\ref{conflaw}.   
Thus all the arguments and results of this section apply without the
restriction $k\leq n/2$ for $n$ even if $g$ is locally
conformally flat.  In particular, this gives another argument for the 
variational character of the $\sigma_k$ in this case.  

The relations
asserted by Theorem~\ref{conflaw}, Proposition~\ref{deltav} and
Theorem~\ref{L} are not obvious for locally conformally flat metrics.
Equation \eqref{confflat} can be written 
\begin{equation}\label{cflat}
g_{ij}(\rho)=g_{ij}(0)+2P_{ij}\rho +P_{ik}P^k{}_j\rho^2.
\end{equation}
Let us set $\gamma = g(0)$ and $A=\gamma^{-1}P$.  Then \eqref{cflat} can be
written $g(\rho)=\ga (I+\rho A)^2$.  Therefore 
\begin{equation}\label{integral}
\int_0^\rho g^{-1}(u)\,du=\int_0^\rho(I+u A)^{-2}\,du\;\ga^{-1}
=\rho(I+\rho A)^{-1}\;\ga^{-1}. 
\end{equation}
Hence
$$
g(\rho)\int_0^\rho g^{-1}(u)\,du = \rho\ga (I+\rho A)\,\ga^{-1}.
$$
Comparing with \eqref{Yformula} gives
\begin{equation}\label{Yflat}
Y_j = -\rho \left(\delta_j{}^k+\rho P_j{}^k\right)\om_k.
\end{equation}
Taking the conformal variation in \eqref{cflat} yields
$$
(\delta g)_{ij} = 2\om \ga_{ij} -2\om_{ij}\rho  
+\left(-2\om P_{ik}P^k{}_j -2\om_{k(i}P_{j)}{}^k\right)\rho^2.
$$
The fact that this agrees with the right hand side of \eqref{newlaw} can be
verified using \eqref{Yflat} and the relation 
between the Levi-Civita connections of $g(\rho)$ and $\ga$ derived in Lemma 
7.3 of \cite{FG2}.  

Similarly, it can be verified directly that if $g$ is locally conformally
flat, 
then \eqref{Lformula} reduces to $L^{ij}_{(k)}=-T^{ij}_{(k-1)}(g^{-1}P)$
for $k\leq n$ and to $0$ for $k>n$, and that \eqref{dvkform} reduces to 
\eqref{deltasigmak}.  The identification of the $L^{ij}_{(k)}$ is clearly  
equivalent to showing that   
$$
v(\rho)\int_0^\rho g^{ij}(u)\,du 
=\sum_{k=1}^n T^{ij}_{(k-1)}(g^{-1}P) \rho^k. 
$$
Now $v(\rho)=\det(I+\rho A)$, so \eqref{integral} shows that this 
can be rewritten as 
$$
\det(I+\rho A)(I+\rho A)^{-1}
=\sum_{k=0}^{n-1} T_{(k)}(A) \rho^k.
$$
It is standard and can be seen in a variety of ways that this 
is a reformulation of the definition of the $T_{(k)}(A)$.  Thus one
concludes that $L^{ij}_{(k)}=-T^{ij}_{(k-1)}(g^{-1}P)$.  Finally,  
\eqref{dvkform} reduces to \eqref{deltasigmak} since 
$\nabla_i\left(T^{ij}_{(k-1)}(g^{-1}P)\right)=0$.  The fact that
$\nabla_i\left(T^{ij}_{(k-1)}(g^{-1}P)\right)=0$
for locally conformally flat metrics follows from the vanishing of the
Cotton tensor and is essentially equivalent to  
the variational characterization of the $\sigma_k$.  See \cite{V} or  
\cite{BG}.  

Theorem~\ref{conflaw} can be used as the basis for another proof of
Theorem~\ref{atmost2}, the fact that under conformal change, the ambient
metric 
coefficients $\pa^k_\rho g_{ij}|_{\rho =0}$ and the $v_k$ depend on at most
second  
derivatives of $\om$.  As observed above, the fact that this is true under  
infinitesimal conformal change is immediate from Theorem~\ref{conflaw}.  
Thus Theorem~\ref{atmost2} follows if we can prove that the
full conformal 
transformation law depends on at most $2$ derivatives of the conformal
factor, assuming that this is the case for the infinitesimal transformation
law.  We formulate a general result along these lines.     

Consider a polynomial natural tensor $T(g)$ depending on a Riemannian
metric in 
dimension $n\geq 2$, of contravariant rank $a$ and covariant rank $b$.  
$T(g)$ may be expressed by evaluating a linear combination of 
partial contractions of covariant indices against contravariant indices 
of $g$, $g^{-1}$, and the
covariant derivatives $\na^{r}R$, $r\geq 0$, of the curvature tensor $R$ of 
$g$.  Each such partial contraction can be written in the form   
\begin{equation}\label{pcontr}
\operatorname{pcontr}\left(\na^{r_1}R\otimes \cdots \otimes \na^{r_M}R
\otimes g\otimes \cdots \otimes g \otimes g^{-1}\otimes \cdots \otimes
g^{-1}\right). 
\end{equation}
Our convention is that the curvature tensor $R$ has contravariant 
rank $0$ and covariant rank $4$.  

We say that $T$ has homogeneity $h\in \R$ if 
$$
T(e^{2\om}g)=e^{h\om}T(g), \qquad \om \in \R.  
$$
We assume throughout that $T$ has a
well-defined homogeneity; this is no loss of generality since a general
natural tensor is the sum of its homogeneous parts in this sense and all of
our considerations respect homogeneity.  
If the contraction \eqref{pcontr} has $P$ factors of $g$,
$Q$ factors of $g^{-1}$, and involves $C$
contractions of a covariant index against a contravariant index, then the 
contravariant rank $a$, covariant rank $b$, and homogeneity $h$ of the
resulting tensor are given by

\[
\begin{split}
a&=2Q-C\\
b&=4M+\sum_{i=1}^M r_i +2P -C\\ 
h&=2(M+P-Q).
\end{split}
\]
In particular, the quantity 
$$
L\equiv 2M+\sum r_i = b-a-h   
$$
is determined just by the rank and homogeneity of $T$.  $L$ is called 
the {\it level} of $T$; it is the total number of derivatives of $g$
occuring in $T$.  Clearly $L\geq 0$.  

If $T$ has homogeneity $h$, the full conformal variation 
$\De T(g,\om)$ of $T$ is defined to be 
$$
\De T(g,\om)\equiv e^{-h\om}T(e^{2\om}g)-T(g) 
$$ 
for smooth $\om$.  Then $\De T(g,\om)$ is a natural tensor depending on $g$ 
and the scalar function $\om$.  It can be obtained by evaluating 
a linear combination of partial contractions of $g$, $g^{-1}$, the 
$\na^rR$ for 
$r\geq 0$, and the covariant derivatives $\na^l \om$, $l\geq 1$, of 
$\om$ with respect to the Levi-Civita connection of $g$, each
contraction of which contains at least one of the $\na^l \om$.  
In this discussion we use a slightly modified infinitesimal conformal   
variation operator $\db$ by subtracting the scaling term from $\de$.
If $T$ is a natural tensor of homogeneity $h$, define  
$$
\db T(g,\om)\equiv \frac{d}{dt} e^{-ht\om}T(e^{2t\om }g)\Big{|}_{t=0}
=\delta T -h\om T.
$$
Then $\db T(g,\om)$ is a natural tensor depending on $g$ and $\om$; it 
is obtained from $\De T(g,\om)$ by keeping only the terms which are linear 
in the derivatives of $\om$.  It is evident 
that when viewed as a function of $g$, $\De T(g,\om)$  
and $\db T(g,\om)$ also have homogeneity $h$:  
$$
\De T(e^{2\Up} g,\om)=e^{h\Up}\De T(g,\om),\qquad
\db T(e^{2\Up} g,\om)=e^{h\Up}\db T(g,\om),\qquad \Up\in\R. 
$$
We may consider the
infinitesimal conformal variation of $\db T$ in $g$: 
$$
\db^2T(g,\om,\Up)\equiv 
\frac{d}{dt} e^{-ht\Up}\db T(e^{2t\Up}g,\om)\Big{|}_{t=0}.
$$
The equality of second mixed partials implies that
\begin{equation}\label{symmetry}
\db^2T(g,\om,\Up) = \db^2T(g,\Up,\om).
\end{equation}

Let $U(g,\om)$ be a polynomial natural tensor depending on $g$ and $\om$
(for example  
$U=\db T$ or $U=\De T$).  For $m\geq 0$, we will say that     
$U$ involves at most $m$ derivatives of $\om$  
if it can be obtained by evaluating a linear combination of 
partial contractions in which only the tensors $\na^l \om$, $1\leq l\leq m$
appear, together with $g$, 
$g^{-1}$, and the $\na^rR$ for $r\geq 0$.  

\begin{proposition}\label{linfull}
Let $m\geq 0$.  If $\db T$ involves at most $m$ derivatives of $\om$,  
then the same is true for $\De T$.  
\end{proposition}
\noindent
The case $m=0$ is the well-known statement that if a natural tensor is
infinitesimally conformally invariant, then it is conformally invariant.
The proof in this case is simpler than in the case $m>0$.  Clearly $\db T$
involves at most $m$ derivatives of $\om$ if and only if $\delta T$
involves at most $m$ derivatives of $\om$.  

\begin{proof}
The proof is by induction on the level $L=b-a-h$ of $T$.    
First consider the case $L=0$.  If a contraction \eqref{pcontr} 
appears in an expression for $T$, then the relation $L=2M+\sum r_i$ 
forces $M=0$.  Thus $T$ is a linear 
combination of partial contractions only of $g$ and $g^{-1}$.  Such a $T$
is conformally invariant, so the desired conclusion is automatic.  

Assume now that the result is true for natural tensors whose level 
$L$ satisfies $L\leq N$ for some $N\geq 0$.  
Suppose $T(g)$ is a natural tensor of some homogeneity $h$ and level
$N+1$, for which $\db T(g,\om)$ involves
at most $m$ derivatives of $\om$.  We can write   
\begin{equation}\label{S}
\db T(g,\om)= \om_{i_1} S_{(1)}^{i_1}(g) 
+\om_{i_1i_2}S_{(2)}^{i_1i_2}(g) +\cdots 
+\om_{i_1\cdots i_m}S_{(m)}^{i_1\cdots i_m}(g),  
\end{equation}
where for $1\leq l\leq m$, $S_{(l)}(g)$ is a natural tensor of homogeneity
$h$ whose covariant 
rank equals that of $T$ and whose contravariant rank is $l$ more than that
of $T$.  Here $\om_{i_1\cdots i_l}$ denotes the components of $\na^l\om$. 
Since the
level of $T$ is $N+1$, it follows that the level of $S_{(l)}$ is $N+1-l\leq 
N$.  
Since for each $l$, the skew-symmetrization of $\na^l\om$ in any 
two indices can be expressed by the Ricci identity in terms of the tensors
$\na^j\om$ with $1\leq j\leq l-2$ and the $\na^rR$, it follows 
inductively that  
$S_{(l)}(g)$ can be taken to be symmetric in the indices $i_1\ldots i_l$. 
Under this condition the $S_{(l)}(g)$ are uniquely determined.  

We claim that each of the  
$\db S_{(l)}(g,\om)$ involves at most $m$ derivatives of $\om$.  
To see this, take the infinitesimal conformal variation of \eqref{S} in
$g$ with respect to a conformal change $\gh_t=e^{2t\Up}g$.  The
infinitesimal conformal variation of the right hand side may be calculated
via the Leibnitz rule.  Each of the terms $\na^l\om$ has a variation
corresponding to the change of the connection.  It is clear that
$\db(\na^l \om)$ involves at most $l$ derivatives of $\om$ and $\Up$.   
Thus it follows that
\[
\begin{split}
\db^2& T(g,\om,\Up)\\
&= \om_{i_1} \db S_{(1)}^{i_1}(g,\Up) 
+\om_{i_1i_2}\db S_{(2)}^{i_1i_2}(g,\Up) +\cdots 
+\om_{i_1\cdots i_m}\db S_{(m)}^{i_1\cdots i_m}(g,\Up) +U(g,\om,\Up),
\end{split}
\]
where $U(g,\om,\Up)$ is a natural tensor depending on $g$, $\om$, and $\Up$   
which involves at most $m$ derivatives of $\om$ and $\Up$.  
Using \eqref{symmetry}, we obtain 
\[
\begin{split}
\om_{i_1} \db S_{(1)}^{i_1}(g,\Up) &
+\om_{i_1i_2}\db S_{(2)}^{i_1i_2}(g,\Up) +\cdots 
+\om_{i_1\cdots i_m}\db S_{(m)}^{i_1\cdots i_m}(g,\Up)
+\widetilde{U}(g,\om,\Up)\\
&=\Up_{i_1} \db S_{(1)}^{i_1}(g,\om) 
+\Up_{i_1i_2}\db S_{(2)}^{i_1i_2}(g,\om) +\cdots 
+\Up_{i_1\cdots i_m}\db S_{(m)}^{i_1\cdots i_m}(g,\om),
\end{split}
\]
where again $\widetilde{U}$ involves at most $m$ derivatives of $\om$ and
$\Up$.  
Since the left hand side involves at most $m$ derivatives of $\om$, the  
same is true of the right hand side.  Therefore this must also hold for
the coefficient of each of the $\Up_{i_1\cdots i_l}$.  Hence 
each of the  
$\db S_{(l)}(g,\om)$ involves at most $m$ derivatives of $\om$ as claimed.   
Thus the 
induction hypothesis applies to each of the $S_{(l)}(g)$, and we deduce
that for $1\leq l\leq m$, $\De S_{(l)}(g,\om)$ involves at most $m$
derivatives of $\om$.   

Next, recall that $\De T$ can be recovered by integrating $\db T$.  
To see this, note first that
\[
\begin{split}
\frac{d}{dt}\left(e^{-ht\om}T(e^{2t\om} g)\right)
=&\frac{d}{ds}\left(e^{-h(t+s)\om}T(e^{2(t+s)\om} g)\right)\Big{|}_{s=0}\\
=&e^{-ht\om}\frac{d}{ds}\left(e^{-hs\om}T(e^{2s\om}e^{2t\om}
g)\right)\Big{|}_{s=0}
=e^{-ht\om}\db T(e^{2t\om}g,\om).
\end{split}
\]
Thus
\begin{equation}\label{integrate}
\begin{split}
\De T(g,\om)=e^{-h\om}T(e^{2\om}g)-T(g)
=&\int_0^1 \frac{d}{dt}\left(e^{-ht\om}T(e^{2t\om} g)\right)dt\\
=&\int_0^1 e^{-ht\om}\db T(e^{2t\om}g,\om)dt.
\end{split}
\end{equation}
Apply \eqref{S} to evaluate $\db T(e^{2t\om}g,\om)$.  The occurrences of
$\na^l\om$, $1\leq l\leq m$ on the right hand side of \eqref{S} now have to
be evaluated using the Levi-Civita connection of $e^{2t\om}g$.  It is clear
that for fixed $t$, each such evaluation gives rise to a natural tensor
depending on $g$ and $\om$ which involves at most $m$ derivatives of 
$\om$.  Likewise, for each $t$ we have 
$$
e^{-ht\om}S_{(l)}(e^{2t\om}g) = S_{(l)}(g) + \De S_{(l)}(g,t\om),
$$
and the right hand side is a family parametrized by $t$ of natural tensors  
depending on $g$ and $\om$ which involves at most $m$ derivatives of $\om$.   
Substituting into \eqref{integrate} and integrating in $t$, it follows that 
$\De T(g,\om)$ involves at most $m$ derivatives of $\om$.  This completes
the induction step.
\end{proof}


\begin{thebibliography}{BEGM}

\bibitem[BG]{BG} T. P. Branson and A. R. Gover, {\it Variational status
of a class of fully nonlinear curvature prescription problems},
Calc. Var. P. D. E. {\bf 32} (2008), 253--262, {\tt arXiv:math/0610773}.    

\bibitem[CF]{CF} S.-Y. A. Chang and H. Fang, {\it A class of variational 
functionals in conformal geometry}, Int. Math. Res. Not. 
(2008), rnn008, 16 pages, {\tt arXiv:0803.0333}.  

\bibitem[FG1]{FG1} C. Fefferman and C. R. Graham, {\it
Conformal invariants,} in {\it The Mathematical Heritage of \'Elie Cartan 
(Lyon, 1984)},  Ast\'erisque, 1985, Numero Hors Serie, 95--116.  

\bibitem[FG2]{FG2} C. Fefferman and C. R. Graham, {\it
The ambient metric}, {\tt arXiv:math/0710.0919}.   

\bibitem[G]{G} C. R. Graham, {\it Volume and area renormalizations for
conformally compact Einstein metrics}, Rend. Circ. Mat. Palermo,
Ser. II, Suppl. {\bf 63} (2000), 31--42, {\tt arXiv:math/9909042}.

\bibitem[GH]{GH} C. R. Graham and K. Hirachi, 
{\it The ambient obstruction tensor and $Q$-curvature}, in 
{\sl AdS/CFT Correspondence: Einstein
Metrics and their Conformal Boundaries}, IRMA Lectures in Mathematics and 
Theoretical Physics {\bf 8} (2005), 59--71, {\tt arXiv:math/0405068}. 

\bibitem[GJ]{GJ} C. R. Graham and A. Juhl, {\it Holographic formula for
$Q$-curvature}, Adv. Math. {\bf 216} (2007), 841--853, {\tt
arXiv:0704.1673}.    

\bibitem[GL]{GL} C. R. Graham and J. M. Lee, {\it Einstein metrics with 
prescribed conformal infinity on the ball},  
Adv. Math.  {\bf 87} (1991), 186--225.

\bibitem[ISTY]{ISTY} C. Imbimbo, A. Schwimmer, S. Theisen and
S. Yankielowicz, {\it Diffeomorphisms and holographic anomalies},
Class. Quant. Grav. {\bf 17} (2000), 1129--1138, 
{\tt arXiv:hep-th/9910267}.

\bibitem[SS]{SS} K. Skenderis and S. N. Solodukin, {\it Quantum effective
action from the AdS/CFT correspondence}, Phys. Lett. {\bf B472} (2000), 
316--322, {\tt arXiv:hep-th/9910023}.   

\bibitem[V]{V} J. A. Viaclovsky, {\it Conformal geometry, contact geometry,  
and the calculus of variations}, Duke Math. J. {\bf 101} (2000), 283--316.

\end{thebibliography}
\end{document}